\documentclass[a4paper,10pt]{article}

\usepackage{amsmath,amssymb,amsthm,graphicx,
	subfig,%showkeys,
verbatim,%refcheck
}

\usepackage{geometry}
\usepackage{tikz}
\usetikzlibrary{shapes.geometric,calc,decorations.markings,math}

\tikzstyle{nodo}=[circle,draw,fill,inner sep=0pt,minimum size=%
1.5mm]
\tikzstyle{infinito}=[circle,inner sep=0pt,minimum size=0mm]

\newcommand\R{{\mathbb R}}

\newcommand\N{{\mathbb N}}
\newcommand\Hmu{{\mathcal{M}_\mu}}

\newcommand\LL{\mathcal L}
\newcommand\DD{\mathcal D}
\newcommand\Sf{\mathcal S}
\newcommand\f{\frac}
\newcommand\dx{{\,dx}}
\newcommand\dy{{\,dy}}

\newcommand\ee{{\mathcal E}}

\newcommand\JJ{{\mathcal J}}
\newcommand\MM{{\mathcal M}}

\newcommand\NN{\mathcal N}

\newcommand\eps{\varepsilon}

\newcommand\NL{{\mathcal N}_\lambda}
\newcommand{\alevel}{{\widetilde J_\lambda}}
\newcommand{\elevel}{{\widetilde E}}

\newcommand{\interv}{{(-\lambda_\Omega, +\infty)}}

\theoremstyle{definition}

\theoremstyle{plain}
\newtheorem{theorem}{Theorem}[section]
\newtheorem{proposition}[theorem]{Proposition}
\newtheorem{lemma}[theorem]{Lemma}
\newtheorem{corollary}[theorem]{Corollary}

\newcounter{ass}
\newtheorem{assumption}[ass]{Assumption}

\theoremstyle{remark}
\newtheorem{remark}[theorem]{Remark}
\newtheorem*{remark*}{Remark}

\theoremstyle{definition}
\newtheorem{definition}[theorem]{Definition}

\date{}

\title{Action versus energy ground states \\in nonlinear Schr\"odinger equations}

\author{Simone Dovetta$^1$, Enrico Serra$^2$, Paolo Tilli$^2$ \\ \ \\
	{\small$^1$ Dipartimento di Scienze di Base ed Applicate per l'Ingegneria, Universit\`a degli Studi di Roma ``La Sapienza"} \\ {\small Via Antonio Scarpa, 14, 00161 Roma, Italy}
	\\ \ \\{\small$^2$Dipartimento di Scienze
		Matematiche ``G.L. Lagrange'', Politecnico di Torino } \\ {\small
		Corso Duca degli Abruzzi, 24, 10129 Torino, Italy}}

\begin{document}

\maketitle

\begin{abstract} We investigate the relations between normalized critical points of the nonlinear Schr\"odinger energy functional and critical points of the corresponding action functional on the associated Nehari manifold.
Our first general result is that the ground state levels are strongly related by the following duality result: the (negative) energy ground state level  is the Legendre--Fenchel transform of the action ground state level. Furthermore, whenever an energy ground state exists at a certain frequency, then all action ground states with that frequency have the same mass and are energy ground states too. We prove that the converse is in general false and that the action ground state  level may fail to be convex. Next we analyze the differentiability of the ground state action level  and we provide an explicit expression involving the mass of action ground states. Finally we show that similar results hold also for local minimizers. 

 \end{abstract}

\noindent{\small AMS Subject Classification: 35Q55, 49J40, 58E30.
}
\smallskip

\noindent{\small Keywords: nonlinear Schr\"odinger, ground states, action, energy, Nehari manifold, mass constraint, minimization}

\section{Introduction and main results}
\label{sec:intro}

This paper is devoted to the relation between \emph{action} ground states and \emph{energy} ground states of the nonlinear Schr\"odinger (NLS) equation
\begin{equation}
	\label{nlseq}
	\Delta u+|u|^{p-2}u=\lambda u\quad\text{in $\Omega$}
\end{equation}
where  $\lambda$ is a real parameter,   $\Omega$ is a (possibly unbounded) open subset of $\R^N$ and
$u\in H_0^1(\Omega)$ (most of our methods and results, however, are rather general and remain 
valid in other settings, such as Neumann boundary conditions, or even the NLS equation
on metric graphs).   Here and throughout the paper, without further warning, we will always
assume that the exponent $p$ satisfies
\begin{equation}\label{rangep}
p\in (2,2^*), \quad 2^*=\frac {2N}{N-2}\quad\text{($2^*=\infty$ if $N=1,2$),}
\end{equation}
which allows for a standard definition of weak solutions.

Since the seminal papers \cite{BPL83, BPL83bis, Est83,EPL82, strauss}, the literature on semilinear scalar field equations (with \eqref{nlseq} as a prototype) has grown enormously and, with no pretence of being exhaustive,
we just refer the reader to the monograph \cite{cazenave} for a comprehensive discussion of the NLS equation on $\R^N$, and for instance to \cite{FS17,FW21,RRS20} (and references therein) for some of the most recent developments.

The existence of positive solutions to \eqref{nlseq} can be addressed by variational methods in
at least two different ways,
either by minimizing the
{\em action functional} $J_\lambda:H_0^1(\Omega)\to\R$
\begin{equation}
	\label{def J}
	J_\lambda (u) := \frac12\|\nabla u\|_{L^2(\Omega)}^2 + \frac\lambda2 \|u\|_{L^2(\Omega)}^2 -\frac1p\|u\|_{L^p(\Omega)}^p
\end{equation}
on the associated 
{\em Nehari manifold}
\begin{equation}
	\label{def NL}
	\begin{split}
	\NL = &\left\{u\in H_0^1(\Omega) \setminus \{0\}\; : \; J_\lambda'(u)u = 0\right\}\\
	=  &\left\{u\in H_0^1(\Omega) \setminus \{0\}\; : \; \|\nabla u\|_{L^2(\Omega)}^2 + \lambda\|u\|_{L^2(\Omega)}^2 = \|u\|_{L^p(\Omega)}^p\right\},
		\end{split}
\end{equation}
or by minimizing the {\em energy functional} $E:H_0^1(\Omega)\to\R$
\begin{equation}
\label{def E}
E(u)= \frac12\|\nabla u\|_{L^2(\Omega)}^2-\frac1p\|u\|_{L^p(\Omega)}^p
\end{equation}
on the manifold of mass-constrained functions 
\begin{equation}
\label{Mmu}
\Hmu=\left\{u\in H_0^1(\Omega)\,:\,\frac12\|u\|_{L^2(\Omega)}^2=\mu\right\}.
\end{equation}

\begin{definition}[Ground states] With the notation introduced above,
\begin{itemize}
\item[1)] given $\lambda\in\R$, a function $u\in \NL$ is called an \emph{action ground state} if 
$J_\lambda(u)=\JJ(\lambda)$, where
\begin{equation}
\label{def JJ}
\JJ(\lambda):=\inf_{v\in\NL}J_\lambda(v),
\end{equation}
and $\JJ(\lambda)$ is called the {\em action ground state level};

\item [2)] given $\mu\geq 0$, a function $u\in \Hmu$ is called an \emph{energy ground state} if 
$E(u)=\ee(\mu)$, where 
\begin{equation}
\label{def EE}
\ee(\mu):=\inf_{v\in\Hmu}E(v),
\end{equation}
and $\ee(\mu)$ is called the {\em energy ground state level}.
\end{itemize}
\end{definition}

It is well known that,
due to
the  form of $J_\lambda$,  
 an \emph{action} ground state will solve \eqref{nlseq} 
  because $u\in \NL$ is a ``natural constraint'' for $J_\lambda$, i.e. 
any constrained critical point of $J_\lambda$ is in fact a genuine critical point in $H^1_0(\Omega)$. 
This approach is quite natural when one is interested in solutions of \eqref{nlseq}
having a prescribed ``frequency'' $\lambda$ (for a general discussion on the method of Nehari manifold see \cite{SW}). 

On the other hand,  an \emph{energy} ground state $u$ of prescribed mass $\mu$ will solve \eqref{nlseq}
with $\lambda$ as a Lagrange multiplier due to the mass constraint.
These solutions are usually referred to as {\em normalized} (or having \emph{prescribed mass}), 
and in this case the frequency $\lambda$ is not known a priori.
Contrary to critical points of the action functional, the analysis of normalized solutions is relatively recent. Starting from the original paper \cite{Jean97}, this topic  is nowadays a well-developed
research line (see for instance \cite{ADST19,AST16,AST17, BV13, BS17,BBJV17,DT19,IM20,NP20,NTV14,PPVV21,PSV20,PV17} and references therein). 

More generally, in both cases, besides ground states one may also look at (constrained) critical
points,   but
in any case
these
two approaches are clearly intertwined, since any critical point 
$u\in\NL$ of $J_\lambda$ is also a critical point of $E$ in $\MM_{\mu}$
(where $\mu$ is the mass of $u$) and, conversely, any critical point $u\in\MM_\mu$ of $E$ 
is also a critical point of $J_{\lambda}$ in $\mathcal{N}_{\lambda}$ (where
$\lambda$ is the Lagrange multiplier of $u$ that pops up in \eqref{nlseq}).

Despite these  relationships, however, the precise interplay between the ``action approach''
and the ``energy approach'' (in particular, 
the question whether an action ground state is necessarily also an energy ground state,
or the other way round, etc.)
has not been thoroughly investigated yet, and the present
paper aims at taking a first step in this direction.  

Our first general result is that the ``ground state levels'' defined in \eqref{def JJ} and \eqref{def EE}
are strongly related by the following duality result.
\begin{theorem}
\label{THM:dualita}
The (negative) energy ground state level $-\,\ee(\mu)$ is the 
Legendre--Fenchel transform of the action ground state level $\JJ(\lambda)$, that is
\begin{equation}
\label{leg_E}
-\ee(\mu)=\JJ^*(\mu):=\sup_{\lambda \in\R}\left(\lambda\mu-\JJ(\lambda)\right)
\qquad\forall \mu\geq 0.
\end{equation}
\end{theorem}
 The fact that \eqref{leg_E} holds for $\mu\geq0$ only is by no means restrictive, as $\JJ^*(\mu)=+\infty$ for every $\mu<0$ (see Remark \ref{rem:m<0} below), whereas $\ee$ is not even defined for negative masses.
Note that \eqref{leg_E} is valid in full generality, regardless for what $\lambda$ or $\mu$ the infima in \eqref{def JJ} and \eqref{def EE} are attained, and even regardless the finiteness of $\ee(\mu)$
(notice that, while at this level of generality $\ee(\mu)$ may take the value $-\infty$, in any
case $\JJ(\lambda)\geq 0$ because $J_\lambda(u)=(\frac 1 2 -\frac 1 p)\Vert u\Vert_{L^p(\Omega)}^p$
for every $u\in \NL$).

When an energy ground state exists, however, it is {\em always} an action ground state. More precisely, we have the following result.
\begin{theorem}
\label{THM:global}
Given $\mu> 0$, assume $u\in \Hmu$ is an energy ground state of mass $\mu$, 
and let 
$\lambda$ be the Lagrange multiplier associated with $u$ in \eqref{nlseq}.
Then $u$ is also an action ground state on $\NL$. Moreover, 
any other action ground state $v\in \NL$ belongs to $\MM_\mu$ (i.e. $v$ has the same mass as $u$),
and $v$ is also an energy ground state on $\MM_\mu$.
\end{theorem}

This reveals a certain rigidity of the variational framework with respect to energy ground states. Indeed, whenever a frequency $\lambda$ pops up as the Lagrange multiplier
of an energy ground state $u$, not only  is
$u$  also an action ground state in $\NL$, but any other action ground state $v\in\NL$
is forced to have the same mass as $u$, and is itself an energy ground state.

In view of Theorem \ref{THM:dualita} (that
entails the concavity of $\ee(\mu)$), it is natural to wonder if the duality between $\JJ$ and $\ee$ can  be reversed, by expressing the transform of $-\ee$ in terms of $\JJ$. Contrary to \eqref{leg_E}, this question is sensitive to the finiteness of $\ee$. For instance, in the $L^2$--supercritical regime $p>2+ 4/N$, since  $\ee(\mu)=-\infty$ for every $\mu>0$ (and  $\ee(0)=0$),
\[
\JJ^{**}(\lambda)= (-\ee)^*(\lambda)=\sup_{\mu\geq 0}\left(\lambda\mu+\ee(\mu)\right)=0\qquad\forall\lambda\in\R\,.
\]
As $\JJ(\lambda)>0$ for certain values of $\lambda$, it is evident that $\JJ^{**}\not\equiv\JJ$ , so that for these values of $p$ the duality in Theorem \ref{THM:dualita} goes in one direction only. As a by--product, this also shows that, in the $L^2$--supercritical regime, $\JJ$ is never a convex function.

In the $L^2$--subcritical and critical regimes, on the contrary, the situation is more involved. In this case, there always exist values of the mass for which $\ee$ is finite. Nevertheless, whether $\JJ$ coincides with $\JJ^{**}$ is not trivial only if $\JJ$ is a continuous function on $\R$. In view of Lemma \ref{lem:J} below (see also Remark \ref{rem:lomega}), this is equivalent to $\JJ(-\lambda_\Omega)=0$, where
\[
\lambda_\Omega:=\inf_{u\in H_0^1(\Omega)}\f{\|\nabla u\|_{L^2(\Omega)}^2}{\|u\|_{L^2(\Omega)}^2}
\]
denotes the bottom of the spectrum of the Dirichlet Laplacian. The validity of $\JJ(-\lambda_\Omega)=0$, without further assumptions on $\Omega$, seems however to be an open problem. 
Anyway, even in this setting it is possible to prove that the duality of Theorem \ref{THM:dualita} does not hold in the opposite direction in full generality.
\begin{theorem}
\label{THM:pasquetta}
Let $p \le 2 +4/N$ and assume that $\Omega$ has finite measure. If for some $\bar\lambda\in\R$
 there exist two action ground states  $v_1,v_2\in\mathcal{N}_{\bar\lambda}$ with different masses, then
\begin{equation}
\label{eq:pasquetta}
\JJ(\bar\lambda)>\JJ^{**}(\bar\lambda)=\sup_{\mu\geq 0}\left(\bar\lambda\mu+\ee(\mu)\right)\,.
\end{equation}
In particular, $\JJ$ is not a convex function.
\end{theorem}
The previous theorem unravels a certain asymmetry between the two variational problems. Indeed,
if some $\bar\lambda$ allows, as above, for two action ground states of different masses,
then  Theorem \ref{THM:global} prevents the existence of any energy ground state (of any
prescribed mass) with frequency $\bar\lambda$. This also shows that the implication of Theorem \ref{THM:global} ``energy ground state $\implies$ action ground state'' cannot be reversed, in general.

a statement concerning {\em action} ground states analogous to Theorem \ref{THM:global} is in general false. 

At present, we do not know any reference in the literature providing a domain $\Omega$ and a frequency $\bar\lambda$ 
satisfying the hypotheses of Theorem \ref{THM:pasquetta}. However, we believe that this may happen,
and we can exhibit an explicit example of such a phenomenon in the context of NLS equations on metric graphs, that will be the object of a forthcoming paper. The finite measure assumption in the preceding result is of course far from  sharp. We stated Theorem \ref{THM:pasquetta} in its present form to highlight the key idea underpinning the possible loss of convexity of $\JJ$ in the most basic framework possible. However more general conditions can be considered to extend the result to sets of infinite measure (for the major differences arising in this case see Remark \ref{rem:infmes} below).

Clearly, though up to now we pursued a wide generality, this type of results is most meaningful in those regimes
where ground states of either kind do exist, which of course depends
on the power $p$, on the values of $\lambda$ and $\mu$ being considered, and on $\Omega$. On the one hand, in light of the above discussion it is obvious that
problem \eqref{def EE} admits no solution (for any $\mu>0$) whenever $p>2+4/N$, while existence of energy ground states at the $L^2$--critical power $p=2+4/N$ strongly depends on the specific value of the mass. On the other hand,
existence of action ground states is possible only when $\lambda$ exceeds $-\lambda_\Omega$.
Therefore we introduce the following

\begin{assumption}\label{assE}
Let

\begin{equation}
\label{asspl}
2<p<2+\frac 4 N,\qquad \lambda>-\lambda_\Omega.
\end{equation}
We assume that $\Omega$ is such that action ground states exist for every $\lambda>-\lambda_\Omega$ and energy ground states exist for every $\mu> 0$.
\end{assumption}

For the sake of clarity, we state our next result under Assumption \ref{assE}, though it remains valid as soon as existence is  known to hold in certain intervals
of frequencies and masses. In Section \ref{sec:app} we will provide concrete classes of domains on which our analysis applies.

Note that the (possible) non--convexity of $\JJ$ is in contrast with
the concavity of $\ee$ which 
 entails
that $\ee(\mu)$ is differentiable except, at worst, for a countable set of masses. 
In \cite{DST20}, we further investigated the differentiability of $\ee$, showing that its right and left derivatives satisfy
\[
\ee_+'(\mu)=-\Lambda_+(\mu),\qquad \ee_-'(\mu)=-\Lambda_-(\mu)\,,
\]
where $\Lambda_+$ and $\Lambda_-$ denote, respectively,  the maximum and minimum frequency 
associated with an energy ground state of mass $\mu$. Although
the results in \cite{DST20} are derived in the framework of metric graphs, 
the methods used therein are general and  cover the case of
equation \eqref{nlseq} in $\Omega\subseteq\R^N$ (under a closure assumption analogous to the one discussed below).

Without relying on convexity, however, we can prove 
similar differentiability properties for the function $\JJ(\lambda)$ as well.
To this end, for $\lambda> -\lambda_\Omega$,
we define the  set    
\begin{equation}
\label{qq}
Q(\lambda) := \left\{\mu \,|\,\, \text{there exists an action ground state $u\in \NL$
with $\f12\|u\|_{L^2(\Omega)}^2=\mu$}
 \right\},
\end{equation}
i.e. the set of masses achieved by all action ground states with frequency $\lambda$,
and we consider

\begin{assumption}
\label{C}
For every pair of sequences $\lambda_n> -\lambda_\Omega$
and $\mu_n\in Q(\lambda_n)$ such that
\[
\lambda_n\to \lambda\in (-\lambda_\Omega,+\infty),\qquad
\mu_n\to\mu\in\R,
\]
there holds $\mu\in \overline{Q(\lambda)}$, where  $\overline{Q(\lambda)}$ denotes the closure of $Q(\lambda)$.
\end{assumption}

Roughly, Assumption \ref{C} provides a minimum of continuity on the parameters sufficient to deal with differentiability issues. As pointed out in Remark \ref{rem:C}, it is a compactness assumption,  weaker than other compactness properties of the set of action ground states in $H_0^1(\Omega)$.

\begin{theorem}
\label{THM:action}
If Assumptions \ref{assE}--\ref{C} hold, then  

\begin{itemize}
\item[(i)] the left and right derivatives of $\JJ$ exist for every $\lambda \in(-\lambda_\Omega,+\infty)$ and
\[
\JJ_-'(\lambda)=\sup Q(\lambda), \qquad \JJ_+'(\lambda) = \inf Q(\lambda);
\]
\item[(ii)] there exists an at most countable set $Z \subset (-\lambda_\Omega,+\infty)$ such that,
  for every $\lambda \in (-\lambda_\Omega,+\infty)\setminus Z$, the set
$Q(\lambda)=\{\mu_\lambda\}$ is a singleton.
 In particular, $\JJ$ is differentiable  in $(-\lambda_\Omega,+\infty) \setminus Z$, where  
\begin{equation}
\label{eq:J'}
\JJ'(\lambda)=\mu_\lambda.
\end{equation}
\end{itemize}
\end{theorem}

\begin{corollary}
\label{cor:ovvio}
If $\bar\lambda$ satisfies the assumptions of Theorem \ref{THM:pasquetta}, then $\JJ$ is not differentiable at $\bar\lambda$.
\end{corollary}

\begin{remark}
Formula \eqref{eq:J'} may look familiar in the light of the by--now standard stability theory for NLS equations \cite{CL,gss87,gss90,ss,w86,w87}. However, the key starting assumption of those works is to consider a $C^1$--curve of solutions to \eqref{nlseq}, parametrized by the frequency $\lambda$, and all the subsequent differentiability and stability properties are given along this curve only. When dealing with ground states, the presence of this  regular curve is not granted in general, unless one already knows something more such as the uniqueness of the solution (as pointed out for instance in \cite[Section 6]{ss}). As  is well--known, uniqueness issues for semilinear elliptic equations are extremely challenging, and very few results are available for positive solutions of \eqref{nlseq} on radial domains only (see the celebrated paper \cite{K} for the case of decaying radial solutions in $\R^N$, as well as \cite{P} and references therein for an overview on the topic). On the contrary, Theorem \ref{THM:action} exploits the minimality of action ground states only and it does not require any further assumption.
\end{remark}

To conclude, we show that the property of energy ground states to be action ground states as well, described in Theorem \ref{THM:global}, has a local counterpart, that we state as our last result. The proof  relies on an explicit comparison between the second derivatives of the action and the energy and, for this reason, the result is valid also for $L^2$--critical and supercritical powers $p\geq2+4/N$. 

\begin{theorem}
\label{THM:local}
Given  $\mu >0$,  let $u \in\Hmu$ be a nondegenerate local minimum for the energy $E$ constrained to $\Hmu$, and  let $\lambda$ denote its frequency as in \eqref{nlseq}.
Then $u$ is a nondegenerate local minimum for the action $J_\lambda$ on $\NL$.
\end{theorem}

Our results are, to the best of our knowledge,  the first insight on the relation between action and energy ground states in full generality. Of course, stronger results than those in Theorem \ref{THM:global} are available on domains where uniqueness is known, but this applies to the case of the
ball and  few other special cases only. We also mention that, combining \cite[Theorem 3]{FSK12} and \cite[Theorem 1.7]{NTV14}, in the $L^2$--supercritical regime $2+4/N<p<2^*$ and
when $\Omega$ is the unit ball, it is possible to show that the action ground state is not a local minimum of the energy in the corresponding mass constrained space. On the one hand, this implies that our Theorem \ref{THM:local} on local minimizers is sharp in general. On the other hand, this does not relate to the comparison between ground states we developed here, since the definition of energy ground states we consider is meaningless when $p>2+4/N$.

\begin{remark}
The space $\MM_\mu$ is usually defined dropping the (inessential) factor $1/2$ in \eqref{Mmu},
but our choice has the advantage of giving a neat  Legendre transform in  
\eqref{leg_E}: without the factor $1/2$ in \eqref{Mmu}, one would obtain an
equivalent relation in \eqref{leg_E}, in terms of suitable rescalings of either
$\JJ$ or $\ee$.
\end{remark}

\begin{remark}
	With the only exception of Theorem \ref{THM:action}, straightforward adaptations of the arguments presented here allow one to recover all the results of the paper for Schr\"odinger equations with combined nonlinearities
	\[
	\Delta u+|u|^{p-2}u+|u|^{q-2}u=\lambda u, \qquad2<p<q<2^*
	\] 
	(see \cite{JJLV20,Killip,LeCoz,soave1,soave2} and references therein for some recent developments on the topic).
\end{remark}
\medskip

After the present work was completed, we became aware of the interesting paper \cite{JS}, where the authors obtain results strongly related to ours in the case $\Omega=\R^N$ but for a wide class of nonlinearities.
\medskip

The paper is organized as follows. Section \ref{sec:global} recalls some preliminary results, establishes some general properties of the level functions and provides the proof of Theorems \ref{THM:dualita}--\ref{THM:pasquetta}. Section \ref{sec:J} discusses the differentiability properties of the action ground state level as stated in Theorem \ref{THM:action}, whereas Section \ref{sec:local} contains the proof of Theorem \ref{THM:local}. Finally, Section \ref{sec:app} provides examples of domains where the results of the paper apply.

\medskip
\noindent \textbf{Notation.} Throughout, we denote by $\|u\|_q$ the $L^q$ norm of $u$, omitting the domain of integration whenever it is clear from the context.

\section{Preliminaries and proof of Theorems \ref{THM:dualita}--\ref{THM:pasquetta}}
\label{sec:global}

We begin our discussion by stating some useful properties of the energy ground state level $\ee$. To this purpose, we recall, for every $p\in[2,2^*)$, the Gagliardo--Nirenberg inequality
\begin{equation}
\label{GN}
\begin{split}
\|u\|_p^p\leq K_p\|u\|_{2}^{p-N\left(\frac{p}{2}-1\right)}\|\nabla u\|_{2}^{N\left(\frac{p}{2}-1\right)}\,,\qquad\forall u\in H_0^1(\Omega),
\end{split}	
\end{equation}
where $K_p$ is the smallest constant that makes the inequality satisfied, that by invariance under dilations of \eqref{GN} is
\[
K_p=\sup_{u\in H^1(\R^N)} \frac{\|u\|_{L^p(\R^N)}^p}{ \|u\|_{L^2(\R^N)}^{p-\alpha}   \|\nabla u\|_{L^2(\R^N)}^\alpha},\qquad \alpha = N\left(\frac{p}2-1\right).
\]
As a consequence, $K_p$ is independent of $\Omega$ (and is not attained unless $\Omega=\R^N$).

The next lemma collects all the properties of $\ee$ we will need. Most of them are well--known and we report them here for the sake of completeness.
\begin{lemma}
\label{lem:E mu}
Let $\ee:[0,+\infty)\to\mathbb{R}$ be the energy ground state level defined in \eqref{def EE}. The following properties hold:
\begin{itemize}
\item[(i)] if $p\in\left(2,2+\frac4N\right)$, then $\ee(\mu)>-\infty$ for every $\mu\geq0$, $\ee$ is concave on $[0,+\infty)$ and  
$\displaystyle\lim_{\mu\to+\infty}\ee(\mu) /\mu=-\infty$;
\item[(ii)] if $p=2+\frac4N$, then
\begin{equation}
	\label{Ecrit}
	\ee(\mu)\begin{cases}
	\geq0 & \text{if }0<\mu<\mu_N\\
	=0 & \text{if }\mu=\mu_N\\
	=-\infty & \text{if } \mu>\mu_N\,,
	\end{cases}
\end{equation}
where
\begin{equation}
	\label{muN}
	\mu_N:=  \frac12 \left( \frac{p}{2K_p} \right)^{N/2} =  \frac12\left(\frac{N+2}{NK_p}\right)^{N/2},
\end{equation}
and $\ee$ is concave on $[0,\mu_N]$;
\item[(iii)] if $p\in\left(2+\frac4N,2^*\right)$, then $\ee(\mu)=-\infty$ for every $\mu>0$;
\item[(iv)] if $p\in\left(2,2+\frac4N\right]$, then for every $\lambda \in\R$
\[
\sup_{\mu\geq0}\left(\lambda\mu+\ee(\mu)\right)=\max_{\mu\geq0}\left(\lambda\mu+\ee(\mu)\right)\,.
\]

\end{itemize}
\end{lemma}
\begin{proof}
The boundedness properties of $\ee$  in $(i)$--$(ii)$--$(iii)$ are standard and follow from \eqref{GN} (see for instance \cite{cazenave} for the case $\Omega=\R^N$, the general case being analogous). The fact that the threshold $\mu_N$ in \eqref{muN} is the same for every open $\Omega\subseteq\R^N$ is clear since $K_p$ in \eqref{GN} does not depend on $\Omega$.

When $p\in\left(2,2+4/N\right)$, to prove that $\ee$ is concave on $[0,+\infty)$ note that, since for every $\mu>0$ and $u\in \MM_1$ the function $\sqrt\mu u$ belongs to $\MM_\mu$, defining $f_u :[0,+\infty) \to \R$ by
\[
f_u(\mu) := E(\sqrt{\mu}u) = \frac{\mu}{2}\|\nabla u\|_2^2 - \frac{\mu^{p/2}}p \|u\|_p^p,
\]
we have
\[
\ee(\mu)=\inf_{u\in \mathcal{M}_1}f_u(\mu).
\]
Since $f_u$ is  concave  on $[0,+\infty)$ for every $u\in \MM_1$, so is $\ee$.
Furthermore, for any fixed $u \in \MM_1$,
\[
\f{\ee(\mu)}{\mu} \le \f{E(\sqrt{\mu}u)}{\mu} = \frac12\|\nabla u\|_2^2 - \frac{\mu^{\frac{p-2}2}}p \|u\|_p^p \to -\infty
\]
as $\mu \to +\infty$, thus concluding the proof of $(i)$. When $p=2+4/N$ the concavity of $\ee$ on $\left[0,\mu_N\right]$ can be shown as for $(i)$.

It remains to prove $(iv)$. If $p\in\left(2,2+4/N\right)$, it is enough to note that $\lambda\mu+\ee(\mu)$ is continuous on $[0,+\infty)$ (because $\ee$ is concave), and by $(i)$
\[
\lambda\mu+\ee(\mu) = \mu (\lambda + \ee(\mu)/\mu) \to -\infty\qquad\text{as }\mu \to +\infty\,.
\]
Similarly, if $p=2+4/N$, then by \eqref{Ecrit} it follows $\lambda\mu+\ee(\mu)=-\infty$ for every $\mu>\mu_N$, so that
\[
\sup_{\mu\geq0}\left(\lambda\mu+\ee(\mu)\right)=\sup_{0\leq\mu\leq\mu_N}\left(\lambda\mu+\ee(\mu)\right),
\]
and we conclude as above.
\end{proof}
\begin{remark}
Relying on \eqref{GN}, the previous proof exploits the homogeneous Dirichlet condition at the boundary of $\Omega$. However, if one is interested in Neumann boundary conditions, the results of Lemma \ref{lem:E mu} can be proved exactly in the same way, considering the corresponding Gagliardo--Nirenberg inequality
\[
\|u\|_{L^p(\Omega)}^p\leq K_{p,\Omega}'\|u\|_{L^2(\Omega)}^{p-N\left(\frac{p}{2}-1\right)}\|u\|_{H^1(\Omega)}^{N\left(\frac{p}{2}-1\right)}\,.
\]
\end{remark}
We now turn our attention to the action ground state level $\JJ$ and to the relations between $\JJ$ and $\ee$. Let us first recall that, for every $u\in\NL$, one can rewrite the action functional $J_\lambda(u)$ as
\begin{equation}
\label{J tilde}
{J_\lambda}(u) = \kappa \|u\|_p^p=:\alevel(u), \qquad \kappa = \frac12 - \frac1p\,,
\end{equation}
so that
\[
\JJ(\lambda)=\inf_{u\in\NL}\alevel(u)\,.
\]
This immediately shows that $\JJ$ is nonnegative. Recall also that, by Sobolev embeddings, for every $\lambda>-\lambda_\Omega$ there exists $\alpha>0$ (depending on $\lambda$) such that $\JJ(\lambda)\geq\alpha$. 

The key point in the comparison between the two minimization problems is the following result.

\begin{proposition}
\label{prop_pasqua}
For every $\lambda\in\R$, $v\in\NN_\lambda$ and $\mu\geq0$, there results	\begin{equation}
\label{Jv geq Ju}
J_\lambda(v)\geq \ee(\mu)+\lambda\mu\,.
\end{equation}
Equality in \eqref{Jv geq Ju} holds if and only if $v\in\Hmu$ and it is both an energy ground state on $\MM_\mu$ and an action ground state on $\NL$.
\end{proposition}
\begin{proof} If $\mu=0$, then \eqref{Jv geq Ju} trivially holds (with strict inequality), as $\ee(\mu)=0$ and $J_\lambda(v)>0$ by \eqref{J tilde}. Let then $\mu>0$ and $v$ be any element in $\NL$, and denote $m: = \|v\|_2^2$. By definition of Nehari manifold, for every $t>0$ there results $J_\lambda(tv) \le J_\lambda(v)$, with strict inequality unless $t=1$. Thus, given any $\mu >0$,

\begin{equation}
\label{ineq1}
J_\lambda(v) \ge J_\lambda\left(\sqrt{\frac{2\mu}{m}}v\right) = E\left(\sqrt{\frac{2\mu}{m}}v\right) + \lambda\mu\ge
\ee(\mu) + \lambda\mu,
\end{equation}
since $\sqrt{2\mu/m}\,v \in \MM_\mu$, and \eqref{Jv geq Ju} is proved.
	
To conclude, note that if $v\in \MM_\mu$ and is an energy ground state, then
\[
J_\lambda(v) = E(v) + \lambda\mu = \ee(\mu) +\lambda\mu
\]
(here the fact that $v$ is an action ground state is not used).
	
Conversely, if equality occurs in \eqref{Jv geq Ju} for some $v\in\NL$, then \eqref{ineq1} is an equality, showing at the same time that $m= 2\mu$, namely $v\in \MM_\mu$, and  that $E(v)=\ee(\mu)$, namely that $v$ is an energy ground state.
Furthermore, $v$ is also a minimizer of $J_\lambda$ in $\NL$, because if this were not the case then there would exist $w\in\NL$, $w\neq v$, satisfying $J_\lambda(w)<J_\lambda(v)=\ee(\mu)+\lambda\mu$, contradicting \eqref{Jv geq Ju}. 
\end{proof}

Relying also on the above proposition, we can now establish the next general properties of $\JJ$. As it will be useful in the following, given $\lambda\in\R$ and $u\in H_0^1(\Omega)$, we set

\begin{equation}
\label{sigma}
\sigma_\lambda(u):=\left(\f{\|\nabla u\|_2^2+\lambda\|u\|_2^2}{\|u\|_p^p}\right)^{\f 1{p-2}}\,.
\end{equation}
Note that $\sigma_\lambda(u)u\in\NL$.

\begin{lemma}
\label{lem:J}
Let $\JJ:\R\to\R$ be the action ground state  level defined in \eqref{def JJ}. The following properties hold:
\begin{itemize}
\item[(i)] $\JJ(\lambda)=0$ for every $\lambda<-\lambda_\Omega$;
\item[(ii)] $\JJ$ is increasing on $\R$ and continuous on $\R\setminus\{-\lambda_\Omega\}$;
\item[(iii)] 
\begin{equation}
\label{J/lambda}
\lim_{\lambda\to+\infty}\f{\JJ(\lambda)}\lambda=\begin{cases}
+\infty & \text{if }p\in\left(2,2+\frac4N\right)\\
\mu_N & \text{if }p=2+\frac4N\\
0 & \text{if }p\in\left(2+\frac4N,2^*\right),
\end{cases}
\end{equation}
where $\mu_N$ is the number defined in \eqref{muN}.
\end{itemize}	
\end{lemma} 

\begin{proof}
The proof is divided into several step.
\medskip

\noindent{\em Step 1: proof of (i).} Since $\lambda<-\lambda_\Omega$, there exists a bounded subset $\Omega'\subset\Omega$ so that $\lambda_{\Omega'}=-\lambda$. Let $\varphi_1,\varphi_2\in H_0^1(\Omega')$ be the eigenfunctions associated to $\lambda_{\Omega'}$ and to the second eigenvalue $\lambda_2$ of the Dirichlet Laplacian on $\Omega'$, respectively. For $\varepsilon>0$, let $v_\varepsilon:=\sigma_\lambda(\varphi_1+\varepsilon\varphi_2)\left(\varphi_1+\varepsilon\varphi_2\right)$. Then $v_\varepsilon\in\NN_\lambda$ by definition of $\sigma_\lambda$ and because $H_0^1(\Omega')\subset H_0^1(\Omega)$. Moreover, recalling \eqref{sigma}, as $\eps \to 0$,

\[
\begin{split}
\sigma_\lambda(\varphi_1+\varepsilon\varphi_2)=&\left(\f{\|\nabla\varphi_1+\varepsilon\nabla\varphi_2\|_2^2+\lambda\|\varphi_1+\varepsilon\varphi_2\|_2^2}{\|\varphi_1+\varepsilon\varphi_2\|_p^p}\right)^{\f1{p-2}}\\
=&\left(\f{\|\nabla\varphi_1\|_2^2+\varepsilon\|\nabla\varphi_2\|_2^2+\lambda\|\varphi_1\|_2^2+\lambda\varepsilon^2\|\varphi_2\|_2^2}{\|\varphi_1+\varepsilon\varphi_2\|_p^p}\right)^{\f1{p-2}}=\left(\f{\varepsilon^2(\lambda_2+\lambda)\|\varphi_2\|_2^2}{\|\varphi_1+\varepsilon\varphi_2\|_p^p}\right)^{\f1{p-2}} = o(1),
\end{split}
\]
where we used the fact that $\varphi_1,\varphi_2$ are orthogonal in $L^2(\Omega)$, $\|\nabla\varphi_1\|_2^2=\lambda_{\Omega'}\|\varphi_1\|_2^2=-\lambda\|\varphi_1\|_2^2$ and $\|\nabla\varphi_2\|_2^2=\lambda_2\|\varphi_2\|_2^2$ by construction. Hence,
\[
0\leq\JJ(\lambda)\leq\lim_{\varepsilon\to0}J_\lambda(v_\varepsilon)=\lim_{\varepsilon\to0}\kappa\sigma_\lambda(\varphi_1+\varepsilon\varphi_2)^p\|\varphi_1+\varepsilon\varphi_2\|_p^p=0\,.
\]
\medskip
	
\noindent{\em Step 2: proof of (ii).} In view of $(i)$ and of the nonnegativity of $\JJ$, it is enough to prove that $\JJ$ is increasing on $[-\lambda_\Omega,+\infty)$ and continuous on $(-\lambda_\Omega,+\infty)$. 
	
Let then $-\lambda_\Omega \le \lambda < \lambda'$. For every $u\in\NN_{\lambda'}$, we see from \eqref{sigma} 
that $\sigma_{\lambda}(u) \le 1$. Therefore
\[
\JJ(\lambda)\leq J_{\lambda}(\sigma_{\lambda}(u)u)=\kappa\sigma_{\lambda}(u)^p\|u\|_p^p
\le \kappa \|u\|_p^p = J_{\lambda'}(u)
\]
Hence, passing to the infimum over $u\in\NN_{\lambda'}$ yields $\JJ(\lambda) \le \JJ(\lambda')$.
	
As for the continuity of $\JJ$, note first that for every $\lambda > -\lambda_\Omega$, by definition of $\NL$, 
\[
\frac{\|u\|_2^2}{\|u\|_p^p} \le \frac{1}{\lambda+\lambda_\Omega}, \qquad \forall u \in \NL.
\]
Now let $\lambda, \lambda' > -\lambda_\Omega$ and for $u \in \NN_{\lambda'}$ notice that 
\[
\JJ(\lambda) \le J_\lambda(\sigma_\lambda(u)u) = \left( 1 + (\lambda-\lambda')\frac{\|u\|_2^2}{\|u\|_p^p}\right)^{\frac{p}{p-2}}\kappa \|u\|_p^p = (1 + o(1))J_{\lambda'}(u)
\]
as $\lambda' \to \lambda$. Passing to the infimum over $u\in \NN_{\lambda'}$ we obtain
\[
\JJ(\lambda) -\JJ(\lambda') \le o(1).
\]
Reversing the role of $\lambda$ and $\lambda'$, we also have $\JJ(\lambda') -\JJ(\lambda) \le o(1)$, and continuity is proved.
\medskip

\noindent{\em Step 3: proof of (iii) for $p\in\left(2,2+4/N\right)$.} For every $\lambda >0$, by passing to the infimum over $v\in\NN_\lambda$ in Proposition \ref{prop_pasqua}, we have
$\JJ(\lambda) \ge \ee(\mu) + \lambda\mu$  for every $\mu >0$. Note that $\ee(\mu)$ is finite since $p\in\left(2,2+4/N\right)$. Therefore
\[
\liminf_{\lambda\to +\infty} \frac{\JJ(\lambda)}{\lambda} \ge \liminf_{\lambda\to +\infty} \frac{\ee(\mu) +\lambda\mu}{\lambda} = \mu.
\]
Since $\mu$ is arbitrary, the conclusion follows.
\medskip
	
\noindent{\em Step 4: proof of (iii) for $p\in\left(2+4/N,2^*\right)$.} Let $B=B_r(x_0)$ be a ball contained in $\Omega$ and take a function $v \in H^1_0(B)$ satisfying $\|\nabla v\|_{L^2(B)}^2 + \|v\|_{L^2(B)}^2 = \|v\|_{L^p(B)}^p$ (namely, $v\in \NN_1(B)$). For every $\lambda \ge 1$, define 
\[
v_\lambda(x) = \lambda^{\frac{1}{p-2}}v(\sqrt\lambda(x-x_0)).
\]
Now, $v_\lambda$ is supported in $B_{r/\sqrt\lambda}(x_0)$ and, after extending it to $0$ outside the ball, we can view it as an element of $H_0^1(\Omega)$. By elementary computations, we see that $v_\lambda \in \NL$ for every $\lambda$. Thus
\[
0 \le \frac{\JJ(\lambda)}{\lambda} \le \frac{J_\lambda(v_\lambda)}{\lambda} = \frac{\kappa}{\lambda} \|v_\lambda\|_p^p = \kappa \lambda^{\frac{p}{p-2}-\frac{N}{2}-1}\|v\|_{L^p(B)}^p.
\]
Since $\frac{p}{p-2}-\frac{N}{2}-1<0$ when $p > 2+ 4/N$, letting $\lambda \to +\infty$, we conclude.
\medskip
	
\noindent{\em Step 5: proof of (iii) for $p=2+4/N$.} On the one hand, by Lemma \ref{lem:E mu}$(ii)$ and Proposition \ref{prop_pasqua} with $\mu=\mu_N$, for every $\lambda\in\R$ we have
\[
\JJ(\lambda)\geq\ee(\mu_N)+\lambda\mu_N=\lambda\mu_N,
\]
yielding $\JJ(\lambda)/\lambda\geq\mu_N$. On the other hand, if $\lambda$ is sufficiently large, there exists $v_\lambda\in\NN_\lambda$, compactly supported in a ball contained in $\Omega$, and such that $\|v_\lambda\|_2^2=\mu_N$ and $E(v_\lambda)=o(1)$ as $\lambda \to +\infty$ (to construct  $v_\lambda$ it is for instance enough to consider suitable compactly--supported truncations of the $L^2$--critical solitons in $\R^N$). Then 
\[
\limsup_{\lambda\to+\infty}\f{\JJ(\lambda)}\lambda\leq \lim_{\lambda\to+\infty}\f{J_\lambda(v_\lambda)}\lambda=\lim_{\lambda\to+\infty}\f{\lambda\mu_N+o(1)}\lambda=\mu_N
\]
and the proof is complete.
\end{proof}

\begin{remark}
\label{rem:lomega}
Note that, adapting the argument in Step 2 of the previous proof, one can show that $\JJ(\lambda)$ is continuous from the right  at $\lambda=-\lambda_\Omega$. Hence, by Lemma \ref{lem:J}$(i)$--$(ii)$, the continuity of $\JJ$ on the whole of $\R$ is  equivalent to $\JJ(-\lambda_\Omega)=0$. This equality can be easily proved (repeating the argument in the proof of Lemma \ref{lem:J}, Step 1) whenever $-\lambda_\Omega$ is attained by a corresponding eigenfunction in $H_0^1(\Omega)$. This is for instance the case if $\Omega$ has finite measure. Another condition sufficient for the continuity of $\JJ$ is the existence of energy ground states $u\in\MM_{\mu}$ for arbitrarily small masses. To see this,
suppose ($u_n)_n$ is a sequence of energy ground states with masses $\mu_n \to 0$ and frequencies $\lambda_n$ (of course
larger than $-\lambda_\Omega$). By \eqref{GN},  $u_n$ is bounded in $H^1$ and, as $\mu_n \to 0$, 
$\|u_n\|_p \to 0$. But $\|u_n\|_p^p = \JJ(\lambda_n)$ by Theorem \ref{THM:global} and
since $\JJ$ is increasing, positive for $\lambda > -\lambda_\Omega$ and $\JJ(\lambda_n) \to 0$, it must be $\lambda \to -\lambda_\Omega$. By continuity, $\JJ(-\lambda_\Omega) = 0$, as claimed.

However, as already anticipated in the Introduction, to prove or disprove the validity of $\JJ(-\lambda_\Omega)=0$ in full generality seems to be an open problem. Incidentally, we observe that the problem is related to certain quantitative versions of the Poincar\'e inequality:  $\JJ$ is continuous (at $-\lambda_\Omega$) if and only if there is no constant $c>0$ such that
\[
\|\nabla u\|_2^2\geq\lambda_{\Omega}\|u\|_2^2+c\|u\|_p^2,\qquad\forall u\in H_0^1(\Omega).
\] 
\end{remark}

Theorems \ref{THM:dualita}--\ref{THM:global}--\ref{THM:pasquetta} are then direct consequences of the above results.

\begin{remark}
	\label{rem:m<0}
	Lemma \ref{lem:J}$(i)$ implies that $\JJ^*(\mu)=+\infty$ for every $\mu<0$, since
	\[
	\JJ^*(\mu)=\sup_{\lambda\in\R}\left(\lambda\mu-\JJ(\lambda)\right)\geq\sup_{\lambda<-\lambda_\Omega}\left(\lambda\mu-\JJ(\lambda)\right)=\sup_{\lambda<-\lambda_\Omega}\lambda\mu = +\infty
	\]
	as soon as $\mu$ is negative. 
\end{remark}

\begin{proof}[Proof of Theorem \ref{THM:dualita}]
When $\mu=0$ the theorem is trivial, as $\ee(\mu)=0$ and $\sup_{\lambda\in\R}\left(-\JJ(\lambda)\right)=0$ by Lemma \ref{lem:J}$(i)$.
Let then $\mu>0$. We split the rest of the proof into three cases, depending on the nonlinearity power.
\smallskip

\noindent{\em Case 1: $p\in\left(2+4/N,2^*\right)$.} In this regime, \eqref{leg_E} plainly holds, since for every $\mu>0$ by Lemma \ref{lem:E mu}$(iii)$ we have $-\ee(\mu)=+\infty$, whereas by Lemma \ref{lem:J}$(iii)$,
\[
\sup_{\lambda\in\R}\left(\lambda\mu-\JJ(\lambda)\right)\geq\lim_{\lambda\to+\infty}\lambda\left(\mu-\f{\JJ(\lambda)}\lambda\right)=+\infty\,.
\]
\smallskip
	
\noindent{\em Case 2: $p\in\left(2,2+4/N\right)$.} In this case, by Lemma \ref{lem:E mu}$(i)$, $-\ee(\mu)<+\infty$. By Proposition \ref{prop_pasqua},
\[
-\ee(\mu)\geq\lambda\mu-\JJ(\lambda)\qquad\forall\lambda\in\R\,,
\]
so that 
\[
-\ee(\mu)\geq\sup_{\lambda\in\R}\left(\lambda\mu-\JJ(\lambda)\right).
\]
Conversely, let $\left(u_n\right)_n\subset \MM_{\mu}$ be such that $\lim_{n\to+\infty}E(u_n)=\ee(\mu)$. Then $u_n\in\NN_{\lambda_n}$, for some $\lambda_n\in\R$, entailing
\[
-\ee(\mu)=-\lim_{n\to+\infty}E(u_n) \le \limsup_{n\to+\infty}\left(\lambda_n\mu-\JJ(\lambda_n)\right)\leq\sup_{\lambda\in\R}\left(\lambda\mu-\JJ(\lambda)\right),
\]
completing the proof of \eqref{leg_E} in the $L^2$--subcritical regime.
\medskip
	
\noindent{\em Case 3: $p=2+4/N$.} In this case we need to argue depending on the value of $\mu$. If $\mu>\mu_N$, then by Lemma \ref{lem:E mu}$(ii)$ we have $-\ee(\mu)=+\infty$, while Lemma \ref{lem:J}$(iii)$ implies
\[
\sup_{\lambda\in\R}\left(\lambda\mu-\JJ(\lambda)\right)\geq\lim_{\lambda\to+\infty}\lambda\left(\mu-\f{\JJ(\lambda)}\lambda\right)=+\infty\,,
\]
and \eqref{leg_E} thus holds. On the contrary, if $\mu\in(0,\mu_N]$, since Lemma \ref{lem:E mu}$(ii)$ ensures that $-\ee(\mu)<+\infty$, it is enough to repeat the argument already developed in the $L^2$--subcritical case.
\end{proof}
\medskip

\begin{proof}[Proof of Theorem \ref{THM:global}]
If $u$ is an energy ground state on $\MM_\mu$ and $u\in\NN_\lambda$, then by \eqref{Jv geq Ju}, for every $w\in\NN_\lambda$,
\[
J_\lambda(w) \ge \ee(\mu) + \lambda \mu = E(u) + \frac12 \lambda\|u\|_2^2 = J_\lambda(u),
\]
namely $u$ is an action ground state on $\mathcal{N}_{\lambda}$. 
Moreover, if $v\in\mathcal{N}_{\lambda}$ is any action ground state on $\mathcal{N}_{\lambda}$, then 
$J_{\lambda}(v) = J_{\lambda}(u)$, i.e. \eqref{Jv geq Ju} is an equality. Proposition \ref{prop_pasqua} then implies that 
$\|v\|_2^2=2\mu$ and $v$ is an energy ground state on $\MM_\mu$. 
\end{proof}
\medskip

\begin{proof}[Proof of Theorem \ref{THM:pasquetta}]
We divide the proof in two cases, dealing separately with the $L^2$--subcritical regime $p\in(2,2+4/N)$ and the critical one $p=2+4/N$.

{\em Case 1: $p\in(2,2+4/N)$}. Note that by \eqref{Jv geq Ju} we already know that for every $\lambda$
\[
\JJ(\lambda)\ge \sup_{\mu\geq0}\left(\lambda\mu+\ee(\mu)\right) = \JJ^{**}(\lambda)\,.
\]
To prove \eqref{eq:pasquetta}, assume by contradiction that $\JJ(\lambda)=\sup_{\mu\geq0}\left(\lambda\mu+\ee(\mu)\right)$ for every $\lambda$, so that in particular  equality holds at $\lambda=\bar\lambda$. Since by Lemma \ref{lem:E mu}$(iv)$ the right--hand side of this equality is attained, there exist $\bar\mu>0$ and an energy ground state $u\in \mathcal{M}_{\bar\mu}$ such that
\begin{equation}
\label{barmu}
\JJ(\bar\lambda)=\max_{\mu>0}\left(\bar\lambda\mu+\ee(\mu)\right)=\bar\lambda\bar\mu+\ee(\bar\mu)=J_{\bar\lambda}(u)\,.
\end{equation}
The existence of the ground state $u$ above is granted by the fact that $\Omega$ is of finite measure, which makes the embedding $H_0^1(\Omega) \hookrightarrow L^p(\Omega)$  compact, for every $p\in[2,2^*)$.
By Proposition \ref{prop_pasqua}, every action ground state in $\mathcal{N}_{\bar\lambda}$ belongs to $\mathcal{M}_{\bar\mu}$, and this contradicts the fact that $v_1$ and $v_2$ have different masses. 
	
Finally, since the finite measure of $\Omega$ implies that $\JJ$ is continuous by Remark \ref{rem:lomega}, if $\JJ$ were convex, then we would have $\JJ(\lambda) =  \JJ^{**}(\lambda)$ for every $\lambda$, and we have just proved that  this is not the case.

{\em Case 2: $p=2+4/N$.} The line of the proof is almost identical to that of the previous case. The only difference is that the finite measure of $\Omega$ implies the existence of energy ground states for every mass $\mu\in(0,\mu_N)$. On the contrary, ground states never exist when $\mu=\mu_N$, since if a ground state at mass $\mu_N$ exists on $\Omega$, then it is also a ground state for the same problem on the whole $\R^N$, but this is impossible if $\Omega\neq\R^N$. Hence, to repeat the argument in the first part of the proof we need to show that the mass $\bar\mu$ realizing \eqref{barmu} is different from $\mu_N$. We do this by proving that for every fixed $\lambda\in\R$ 
\begin{equation}
\label{no muN}
\max_{\mu\in[0,\mu_N]}\left(\lambda\mu+\ee(\mu)\right)>\lambda\mu_N+\ee(\mu_N)=\lambda\mu_N\,.
\end{equation}
As $p= 2 + 4/N$, the Gagliardo-Nirenberg inequality reads
\[
\|u\|_p^p\leq K_p\|u\|_{2}^{4/N}\|\nabla u\|_{2}^2.
\]
For $\mu < \mu_N$, let $u_\mu$ be an energy ground state in $\MM_\mu$. By the preceding inequality, keeping in mind \eqref{muN}, 
\begin{align*}
\ee(\mu) &= E(u_\mu) \ge \frac12 \|\nabla u_\mu\|_2^2
\left(1 -\frac{\mu^{2/N}}{\mu_N^{2/N}}\right) =  \frac12 \|\nabla u_\mu\|_2^2 \left(\frac{\mu_N^{2/N}-\mu^{2/N}}{\mu_N^{2/N}} \right) \\
&=   \frac12 \|\nabla u_\mu\|_2^2(\mu_N-\mu) \left(\frac{2}{N\mu_N} +o(\mu_N-\mu)\right)
\end{align*}
as $\mu \to \mu_N$, namely
\[
\frac{\ee(\mu)}{\mu_N-\mu} \ge C \|\nabla u_\mu\|_2^2.
\]
Now $\|\nabla u_\mu\|_2$ cannot be bounded as $\mu \to \mu_N$, since by the compactness of the embedding $H_0^1(\Omega)\hookrightarrow L^p(\Omega)$ for every $p\in[2,2^*)$  this would imply the existence of an energy ground state in $\MM_{\mu_N}$, which is impossible. Therefore the left--hand side tends to $+\infty$ as $\mu \to \mu_N$. In other words, for every $\lambda \in \R$, there exists $\mu \in (0,\mu_N)$ such that 
\[
\ee(\mu) > \lambda (\mu_N -\mu),
\]
which is what we claimed.
\end{proof}

\begin{remark}
	\label{rem:infmes}
	As the above proof plainly shows, working with sets of finite measure ensures both that $\JJ$ is continuous on  $\R$ and that there always exists an energy ground state with mass $\bar\mu$ realizing the maximum in \eqref{barmu}. Clearly, this assumption can be relaxed, but then to recover the results of Theorem \ref{THM:pasquetta} seems to require a careful analysis of specific properties of the domain under exam. For a glimpse on how the situation becomes more involved, note on the one hand that $\JJ$ and $\JJ^{**}$ always coincide if $\Omega$ contains balls of arbitrary radius (the function $\JJ$ coincides with $\JJ_{\R^N}$ and  as  $\JJ_{\R^N}$ is convex, so is $\JJ$).	
On the other hand, Section \ref{subsec:unbound} below provides nontrivial examples of domains with infinite measure where Theorem \ref{THM:pasquetta} can be recovered by a simple adaptation of the previous argument.
\end{remark}
\section{Proof of Theorem \ref{THM:action}}
\label{sec:J}

This section is devoted to the differentiability properties of $\JJ$ as stated in Theorem \ref{THM:action}. 
\begin{proof}[Proof of Theorem \ref{THM:action}] We consider the following auxiliary problem. Let 
\[
\Sf = \left\{v \in H^1_0(\Omega)\; :\; \|v\|_p =1\right\}
\] 
and, for every $v\in \Sf$, define $h_v: \interv \to \R$ by 
\[
h_ v(\lambda) := \|\nabla v\|_2^2 + \lambda \|v\|_2^2\,.
\]
Let then $h : \interv \to \R$ be
\begin{equation}
\label{acca}
 h(\lambda) := \inf_{v \in S} h_v(\lambda).
\end{equation}
In view of \eqref{sigma},  $u:= \sigma_\lambda(v)v \in \NL$ and
\begin{equation*}
J_\lambda(u) = J_\lambda( \sigma_\lambda(v)v) = \kappa\sigma_\lambda(v)^2 \left( \|\nabla v\|_2^2 + \lambda \|v\|_2^2 \right) = \kappa h_v(\lambda)^{\frac{2}{p-2}}h_v(\lambda) =  \kappa h_v(\lambda)^{\frac{p}{p-2}},
\end{equation*}
so that passing to the infimum over $v\in\mathcal{S}$
\begin{equation}
\label{J h}
\JJ(\lambda) = \kappa h(\lambda)^{\frac{p}{p-2}}\,.
\end{equation}
Thus, minimizers in $\Sf$ of $h(\lambda)$ and action ground states in $\NL$ are in one--to--one correspondence and, in particular, by Assumption \ref{assE} one can replace the infimum by the minimum in \eqref{acca}. Also, recalling \eqref{qq}, if $m$ is the mass of a minimizer for $h(\lambda)$, then $m=2 h(\lambda)^{\frac{2}{2-p}}\mu$, where $\mu \in Q(\lambda)$, namely $2\mu$ is the mass of an action ground state in $\NL$.

The function $h(\lambda)$, being the minimum of the affine functions $h_v(\lambda)$, is concave and, as such, it has right and left derivatives everywhere in $\interv$ and is differentiable outside of an at most countable set $Z$.

If $\lambda$ is a point where $h$ is differentiable and $v\in \Sf$ is such that $h_v(\lambda) = h(\lambda)$ and $\|v\|_2^2=m$, then clearly $h'(\lambda) = m =2 h(\lambda)^{\frac{2}{2-p}}\mu$ for some $\mu \in Q(\lambda)$. Incidentally, this also shows that
$h$ is differentiable at $\lambda$ if and only if $Q(\lambda)$ is a singleton.

We now assume $\lambda \in Z$ and we compute the right derivative $h_+'(\lambda)$. To this aim, let $v$ be any element of $\Sf$ such that $h_v(\lambda) = h(\lambda)$, and assume that it has mass $m = 2h(\lambda)^{\frac{2}{2-p}}\mu$ for some $\mu \in Q(\lambda)$. Then $h_+'(\lambda) \le m$, and since this happens for every minimizer of $h(\lambda)$,
\[
h_+'(\lambda) \le \inf_{\mu \in Q(\lambda)} 2h(\lambda)^{\frac{2}{2-p}}\mu =2 h(\lambda)^{\frac{2}{2-p}}\mu^-(\lambda).
\]
To prove the reversed inequality, let $(\lambda_n)_n$ be a sequence such that $h$ is differentiable at every $\lambda_n$, $\lambda_n> \lambda$ for every $n$ and $\lambda_n \to \lambda$ as $n \to \infty$. Letting $h'(\lambda_n) =m_n$, by the concavity of  $h$ we have for every $n \in \mathbb{N}$
\begin{equation}
\label{dis1}
h_+'(\lambda) \ge h_+'(\lambda_n) = h'(\lambda_n) = m_n= 2h(\lambda_n)^{\frac{2}{2-p}}\mu_n, 
\end{equation} 
for some $\mu_n \in Q(\lambda_n)$.
Since the sequence $m_n$ is bounded, we can assume without loss of generality that it converges to some $m$, and hence also $\mu_n$ converges to some $\mu$, as $h$ is continuous. We thus have sequences $\lambda_n \to \lambda$ and $(\mu_n)_n \subset Q(\lambda_n)$ with $\mu_n \to \mu$. By Assumption \ref{C}, $\mu \in \overline{Q(\lambda)}$, so that $\mu \ge \inf Q(\lambda) \ = \mu^-(\lambda)$. Letting $n\to \infty$ in \eqref{dis1}, we obtain
\[
h_+'(\lambda) \ge 2 h(\lambda)^{\frac{2}{2-p}}\mu \ge 2  h(\lambda)^{\frac{2}{2-p}}\mu^-(\lambda).
\]
The computation of the left derivative $h_-'(\lambda)$ is completely analogous.

Finally, recalling \eqref{J h}, we see that $\JJ$ is differentiable in $\interv \setminus Z$, where
\[
\JJ'(\lambda) = \kappa \frac{p}{p-2} h(\lambda)^{\frac{2}{p-2}}h'(\lambda) = h(\lambda)^{\frac{2}{p-2}}h(\lambda)^{\frac{2}{2-p}}\mu = \mu,
\]
where $2\mu$ is the mass of any action ground state in $\NL$ (recall that $Q(\lambda)$ is a singleton). The same computation works for the left and right derivatives, obtaining respectively $\JJ_-'(\lambda) = \mu^+(\lambda)$ and $\JJ_+'(\lambda) = \mu^-(\lambda)$ for every $\lambda \in Z$.
\end{proof}

\begin{proof}[Proof of Corollary \ref{cor:ovvio}]
If $\bar\lambda$ is such that there exist two action ground states $v_1,v_2\in\mathcal{N}_{\bar\lambda}$ with $\|v_1\|_2\neq\|v_2\|_2$, then $\mu^-(\bar\lambda) < \mu^+(\bar\lambda)$. Therefore, Theorem \ref{THM:action} yields $\JJ_-'(\bar\lambda)>\JJ_+'(\bar\lambda)$ and $\JJ$ is not differentiable at $\bar\lambda$.
\end{proof}

\section{Local minima}
\label{sec:local}

This section is devoted to the relation between local minima of the energy functional   and local minima of the action functional, namely Theorem \ref{THM:local}. Its proof  relies on the following propositions that compute and compare the second derivatives of the functionals.

We start by considering a critical point $u$ of $J_\lambda$ on $\NL$. The function $u$ has a certain mass $2\mu =\|u\|_2^2$, and is therefore a critical point of $E$ on $\MM_\mu$. Equivalently, we could start from a critical point $u$ of $E$ on $\MM_\mu$ and view it as a critical point of $J_\lambda$ on $\NL$, where $\lambda$ is the Lagrange multiplier associated with $u$.
To avoid confusion we denote by $\alevel$ the functional $J_\lambda$ considered as a map from $\NL$ to $\R$ and, similarly, $\elevel$ will denote  $E$ restricted to $\MM_\mu$.

We begin by computing the second derivative of $\alevel$ at $u$. 
Note that the tangent space to $\NL$ at $u$ is 

\[
T_{u} \NL = \left\{v \in H^1_0(\Omega) \; :\; \int_\Omega |u|^{p-2}u v \dx = 0 \right\}.
\]

\begin{proposition}
	\label{secondera}
	There results 
	\begin{equation}
	\label{secdera}
	\alevel''(u)v^2 =  \int_\Omega|\nabla v|^2 \dx +\lambda  \int_\Omega |v|^2\dx- (p-1)\int_\Omega |u|^{p-2}v^2 \dx,
	\end{equation}
	for every $v\in T_{u} \NL$.
\end{proposition}

\begin{proof} Let $v \in T_{u} \NL$. If  $\gamma : (-\delta,\delta) \to \NL$ is a smooth curve such that
	\[
	\gamma(0) = u,\qquad  \gamma'(0) = v,
	\]
	then 
	\begin{equation}
	\label{secderuno}
	\alevel''(u) v^2 = \frac{d^2}{dt^2} J_\lambda(\gamma(t))\Big|_{t=0}= J_\lambda''(u)[ \gamma'(0)^2] + J_\lambda'(u)  \gamma''(0).
	\end{equation}
	To carry out the computation, we set
	\begin{equation}
	\label{frac}
	g(t):= \sigma_\lambda(u+tv) = \left(\frac{\int_\Omega |\nabla u+t\nabla v|^2\dx + \lambda\int_\Omega |u+tv|^2\dx}{\int_\Omega |u+tv|^p\dx}\right)^{\frac1{p-2}},
	\end{equation}
	and we define, for $\delta$ small, $\gamma : (-\delta,\delta) \to \NL$ as
	\begin{equation}
	\label{gamma}
	\gamma(t) = g(t)(u+tv).
	\end{equation}
	Note that $\gamma(0) =u$. Denoting by $N(t)$ and $D(t)$ the numerator and the denominator in the right--hand side of \eqref{frac}, we see that
	\[
	g'(t)= \frac1{p-2}\left(\frac{N(t)}{D(t)}\right)^{\frac{3-p}{p-2}}\frac{N'(t)D(t)-N(t)D'(t)}{D(t)^2}.
	\] 
	Now $N(0) = D(0)$ because $u \in \NL$, and,  since $u$ is a critical point of $J_\lambda$,   
	\[
	N'(0) = 2\int_\Omega \nabla u\cdot \nabla v\dx +2\lambda \int_\Omega uv\dx = 2\int_\Omega |u|^{p-2}uv\dx =0,
	\]
	because $v \in T_u\NL$. Likewise, $D'(0) = p\int_\Omega |u|^{p-2}uv = 0$. Hence
	\[
	g'(0) =  \frac1{p-2}\left(\frac{N(0)}{D(0)}\right)^{\frac{3-p}{p-2}}\frac{N'(0)D(0)-N(0)D'(0)}{D(0)^2} =0,
	\]
	from which, differentiating \eqref{gamma}, we obtain
	\[
	\gamma'(0) = g'(0)u +g(0)v = v,\qquad \gamma''(0) = g''(0)u+2g'(0)v= g''(0)u.
	\] 
	Thus \eqref{secderuno} becomes
	\[
	\alevel''(u) v^2 =  J_\lambda''(u)[ \gamma'(0)^2] + J_\lambda'(u)  \gamma''(0) = J_\lambda''(u)v^2 + g''(0) J_\lambda'(u)u = 
	J_\lambda''(u)v^2
	\]
	again because $u\in \NL$. Computing the last term we finally obtain
	\[
	\alevel''(u) v^2 =  \int_\Omega|\nabla v|^2 \dx +\lambda  \int_\Omega |v|^2\dx- (p-1)\int_\Omega |u|^{p-2}v^2 \dx,\quad\text{ for every $v\in T_{u} \NL$}.\qedhere
	\]
\end{proof}

Similarly, we can compute the second derivative of the energy $E$, considered as a functional $\elevel$ on the manifold $\Hmu$, at the same $u$ as above. The tangent space to $\Hmu$ at $u$ is 

\[
T_{u} \Hmu= \left\{v \in H^1(\Omega) \; :\; \int_\Omega uv \dx = 0 \right\}.
\]

\begin{proposition}
	\label{secondere}
	There results 
	\begin{equation}
	\label{secdere}
	\elevel''(u)v^2 =  \int_\Omega|\nabla v|^2 \dx +\lambda  \int_\Omega |v|^2\dx- (p-1)\int_\Omega |u|^{p-2}v^2 \dx,
	\end{equation}
	for every $v\in T_{u} \Hmu$.
\end{proposition}

\begin{proof} Working as in the previous proof, we fix $v\in T_{u} \Hmu$
	and we define, for $\delta$ small, a smooth curve $\eta : (-\delta,\delta) \to \Hmu$ as
	\[
	\eta(t) = \frac{\sqrt{2\mu}}{\|u+tv\|_2}(u + tv)=: h(t)(u + tv).
	\]
	Note that $\eta (0) = u$. Differentiating $h$ yields
	\[
	h'(t) = -\sqrt{2\mu}\left(\int_\Omega |u+tv|^2\dx\right)^{-3/2}\int_\Omega (u+tv)v\dx
	\]
	and
	\[
	h''(t) = 3\sqrt{2\mu}\left(\int_\Omega |u+tv|^2\dx\right)^{-5/2}\left(\int_\Omega (u+tv)v\dx\right)^2 -\sqrt{2\mu}\left(\int_\Omega |u+tv|^2\dx\right)^{-3/2}\int_\Omega |v|^2\dx.
	\]
	Now, since $v\in  T_{u} \Hmu$, $\int_\Omega uv\dx = 0$, so that
	\[
	h'(0) = 0,\qquad h''(0) = -\frac1{2\mu}  \int_\Omega |v|^2\dx.
	\]
	Thus, differentiating $\eta$ we have
	\[
	\eta'(0)  =h'(0)u + h(0)v = v,\qquad \eta''(0) = h''(0)u  + 2h'(0)v = -\left(\frac1{2\mu}  \int_\Omega |v|^2\dx\right) u
	\]
	and 
	\begin{align*}
	\elevel''(u) v^2 &= \frac{d^2}{dt^2} E (\eta(t))\Big|_{t=0}= E''(u)[ \eta'(0)^2] + E'(u)  \eta''(0)=
	E''(u) v^2 - \left(\frac1{2\mu}  \int_\Omega |v|^2\dx\right) E'(u)u \\
	&= \int_\Omega |\nabla v|^2\dx - (p-1)\int_\Omega |u|^{p-2}v^2\dx -  \frac1{2\mu}  E'(u)u\int_\Omega |v|^2\dx .
	\end{align*}
	As $E'(u)u = -\lambda \int_\Omega |u|^2\dx = -2\lambda\mu$, we obtain \eqref{secdere}.
\end{proof}
Having determined the second derivatives of $\alevel$ and $\elevel$ at $u$, we can now compare them.

\begin{proposition}
	\label{compare} For every $v \in T_u\NL$ there exists $\varphi \in T_u\Hmu$ such that
	\begin{equation}
	\label{bind}
	\alevel''(u)v^2 = (p-2)\frac{\left( \int_\Omega |u|^{p-2}u\varphi\dx\right)^2}{ \int_\Omega |u|^p\dx} + \elevel''(u)\varphi^2.
	\end{equation}
	In particular, there exists a constant $C=C(u)>0$ such that
	\begin{equation}
	\label{eigen}
	\inf_{v \in T_u\NL \setminus\{0\}} \frac{\alevel''(u)v^2}{ \|v\|_2^2} \ge  C\inf_{\psi \in T_u \Hmu\setminus\{0\}}\frac{\elevel''(u)\psi^2}{ \|\psi\|_2^2}.
	\end{equation}
\end{proposition}

\begin{proof} For every  $v \in T_u\NL$, there exist $\alpha\in\R$ and $\varphi \in  T_u\Hmu$ such that $v = \alpha u + \varphi$. Indeed it is sufficient to write
	\begin{equation}
	\label{uphi}
	v= \frac{\int_\Omega uv\dx}{\int_\Omega u^2\dx} \;u + \left(v - \frac{\int_\Omega uv\dx}{\int_\Omega u^2\dx}\; u\right) =:\alpha u + \varphi
	\end{equation}
	and note that $\varphi   \in  T_u\Hmu$ since $\int_\Omega\varphi u \dx =0$. We now insert this expression for $v$ in \eqref{secdera} and compute
	\begin{align}
	\label{bridge}
	\alevel''(u)v^2 &= \alevel''(u)(\alpha u + \varphi)^2 \nonumber
	\\&= \int_\Omega |\alpha \nabla u + \nabla \varphi|^2\dx + \lambda \int_\Omega |\alpha u + \varphi|^2\dx
	-(p-1) \int_\Omega |u|^{p-2} |\alpha u + \varphi|^2\dx \nonumber\\
	&= \alpha^2 \left( \int_\Omega |\nabla u|^2 + \lambda |u|^2 -(p-1)|u|^p\dx\right) + 
	2\alpha \left( \int_\Omega \nabla u\cdot \nabla\varphi + \lambda u\varphi -(p-1) |u|^{p-2}u\varphi\dx\right)\nonumber \\
	&+ \int_\Omega |\nabla\varphi|^2 + \lambda |\varphi|^2 -(p-1)|u|^{p-2} \varphi^2\dx.
	\end{align}
	Now $u\in\NL$ and is a critical point of $\alevel$. Therefore
	\[
	\int_\Omega |\nabla u|^2 + \lambda |u|^2\dx=  \int_\Omega |u|^p\dx, \qquad \int_\Omega \nabla u\cdot\nabla\varphi + \lambda u\varphi \dx=  \int_\Omega |u|^{p-2}u\varphi\dx.
	\]
	Noticing also that the last line of \eqref{bridge} is $\elevel''(u)\varphi^2$, we can rewrite \eqref{bridge} as
	\[
	\alevel''(u)v^2 = \alpha^2(2-p) \int_\Omega |u|^p\dx + 2\alpha(2-p)  \int_\Omega |u|^{p-2}u\varphi\dx + \elevel''(u)\varphi^2.
	\]
	But since $v = \alpha u + \varphi$, we see that $\int_\Omega |u|^{p-2}u\varphi \dx = \int_\Omega |u|^{p-2}uv \dx -\alpha\int_\Omega |u|^p\dx= -\alpha \int_\Omega |u|^p \dx$, which plugged in the previous equality yields
	\[
	\alevel''(u)v^2 = \alpha^2(p-2) \int_\Omega |u|^p\dx  + \elevel''(u)\varphi^2.
	\]
	Finally, multiplying \eqref{uphi} by $|u|^{p-2}u$ and integrating we see that
	\begin{equation}
	\label{alfa2}
	\alpha = -\frac{\int_\Omega |u|^{p-2}u \varphi \dx}{\int_\Omega |u|^p\dx}\,,
	\end{equation}
	which combined with the previous equality gives \eqref{bind}.
	
	To prove the second part, we observe that from \eqref{bind}
	\[
	\alevel''(u)v^2 \ge  \|\varphi\|_2^2 \frac{\elevel''(u)\varphi^2}{\|\varphi\|_2^2}\ge  \|\varphi\|_2^2 \inf_{\psi \in T_u \Hmu\setminus\{0\}}\frac{\elevel''(u)\psi^2}{ \|\psi\|_2^2},
	\]
	so that, recalling \eqref{uphi},
	\begin{equation}
	\label{st1}
	\frac{\alevel''(u)v^2}{ \|v\|_2^2} \ge \frac{ \|\varphi\|_2^2}{ \|\alpha u +\varphi\|_2^2} \inf_{\psi \in T_u \Hmu\setminus\{0\}}\frac{\elevel''(u)\psi^2}{ \|\psi\|_2^2}.
	\end{equation}
	Therefore it is sufficient to show that the quantity $\frac{ \|\varphi\|_2^2}{ \|\alpha u +\varphi\|_2^2}$ is uniformly bounded away from zero. 
	
	Since $\varphi \in T_u\MM_\mu$,
	\[
	\frac{ \|\varphi\|_2^2}{ \|\alpha u +\varphi\|_2^2} = \frac{ \|\varphi\|_2^2}{ \|\varphi\|_2^2 +\alpha^2  \|u\|_2^2} =
	\frac{ 1}{1+ \alpha^2 \|u\|_2^2/\|\varphi\|_2^2}
	\]
	and we only have to show that $\alpha^2 \|u\|_2^2/\|\varphi\|_2^2$ is uniformly bounded from above. As $u$ is a critical point of $J_\lambda$, it solves the corresponding nonlinear Schr\"odinger equation \eqref{nlseq} and, by standard elliptic estimates, $u \in L^q(\Omega)$ for every $q \ge 2$. Thus, by \eqref{alfa2},
	\[
	\alpha^2 \frac{\|u\|_2^2}{\|\varphi\|_2^2} = \frac{\left(\int_\Omega |u|^{p-2}u \varphi \dx\right)^2}{\left(\int_\Omega |u|^p\dx\right)^2}\frac{\|u\|_2^2}{\|\varphi\|_2^2} \le \frac{\|u\|_{2p-2}^{2p-2} \|\varphi\|_2^2\|u\|_2^2}{\|u\|_p^{2p} \|\varphi\|_2^2 } =
	\frac{\|u\|_{2p-2}^{2p-2}\|u\|_2^2}{\|u\|_p^{2p} } =:C_1
	\]
	for every $\varphi \in T_u\MM_\mu$.
	
	So,  from \eqref{st1}, we can write
	\[
	\frac{\alevel''(u)v^2}{ \|v\|_2^2} \ge C \inf_{\psi \in T_u \Hmu\setminus\{0\}}\frac{\elevel''(u)\psi^2}{ \|\psi\|_2^2}, \quad C= \frac1{1+C_1},
	\]
	and taking the infimum with respect to $v$, \eqref{eigen} follows.
\end{proof}

\begin{proof}[Proof of Theorem \ref{THM:local}] It is a straightforward consequence of \eqref{eigen}.
\end{proof}

\begin{remark} Formula \eqref{bind} seems to suggest that degenerate (local) minimizers of the energy could be nondegenerate as  critical points of the action, due to the presence of the first term in the right--hand side of \eqref{bind}. This in general is false. For example, if $\Omega =\R^N$, then the energy ground states of mass $\mu$ are the family of solitons $\phi_y=\phi(\,\cdot-y\,)$, where $y\in \R^N$ and $\phi$ is the soliton of mass $\mu$ centered at the origin. The solitons $\phi_y$ are also the action ground states for $\alevel$ on $\NL$ with $\lambda = \LL(\phi)$. So each $\phi_y$ is a degenerate minimum for $\elevel$ on $\Hmu$ {\em and also} a degenerate minimum for $\alevel$ on $\NL$.
\end{remark}

\section{Applications}
\label{sec:app}
This final section discusses the possible application the preceding results  both to bounded and to unbounded domains of $\R^N$. In particular, we provide examples of domains where $\JJ$ is a continuous function on $\R$ and Assumptions \ref{assE}--\ref{C} are fulfilled.

\subsection{Bounded domains}
\label{subsec:bound}

If $\Omega\subset\R^N$ is bounded, then it is readily seen that when $p\in(2,2+4/N)$ there exist both an action ground state in $\NL$ for every $\lambda\in(-\lambda_\Omega,+\infty)$ and an energy ground state in $\MM_\mu$ for every $\mu\in(0,+\infty)$, so that Assumption \ref{assE} is satisfied. By Remark \ref{rem:lomega}, we also have that $\JJ$ is continuous on $\R$. Moreover, the validity of Assumption \ref{C} is straightforward.

\begin{proposition}
\label{compact}
Assumption \ref{C} holds.
\end{proposition}
\begin{proof}
		Let $\lambda\in(-\lambda_\Omega,+\infty)$ be given and let $(\lambda_n)_n$ be a sequence satisfying $\lambda_n \to \lambda$ as $n \to \infty$. Let also $\mu_n \in Q(\lambda_n)$ be such that $\mu_n \to \mu$. By definition, for every $n$ there exists an action ground state $u_n \in \NN_{\lambda_n}$ such that $\|u_n\|_2^2 = 2\mu_n$. The uniform boundedness of $\JJ$ on bounded sets of $\interv$ implies that $(u_n)_n$ is bounded in $H_0^1(\Omega)$. Hence, there exists $u\in H_0^1(\Omega)$ such that $u_n\rightharpoonup u$ in $H_0^1(\Omega)$ and $u_n\to u$ in $L^q(\Omega)$, for every $q\in[2,2^*)$. On the one hand, by the continuity of $\JJ$ one has $\JJ(\lambda)=\kappa\|u\|_p^p$. On the other hand, by weak lower semicontinuity, $\sigma_\lambda(u)\leq1$. Moreover, if $\sigma_\lambda(u)<1$, then 
\[
\JJ(\lambda)\leq \JJ_\lambda(\sigma_\lambda(u)u)=\kappa\sigma_\lambda(u)^p\|u\|_p^p<\kappa\|u\|_p^p=\JJ(\lambda),
\]
a contradiction. Thus $\sigma_\lambda(u)=1$, that is, $u$ is an action ground state in $\NL$. Since
\[
\mu = \lim_n \mu_n = \lim_n\f{\|u_n\|_2^2}2 = \f{\|u\|_2^2}2,
\]
we see that $\mu \in Q(\lambda)$.
\end{proof}

\subsection{Unbounded domains}
\label{subsec:unbound}

If $\Omega$ is unbounded, it can be easily shown that in general there is no action ground state in $\NL$, for any $\lambda$.
Indeed, since  $H^1_0(\Omega) \subset H^1(\R^N)$ (after extending functions to $0$ outside $\Omega$), we immediately see that for every $\lambda >0$
\begin{equation}
\label{RN}
\JJ_{\R^N}(\lambda) \le \JJ(\lambda).
\end{equation}
Here $\JJ_{\R^N}$ denotes the minimum of the action on the associated Nehari manifold in $H^1(\R^N)$.
Suppose now that the unbounded set $\Omega$ contains balls of arbitrary radius. Then one easily sees that \eqref{RN} is an equality, so that  if $\JJ(\lambda)$ is attained by some $u \in H^1_0(\Omega)$, then $u$ is also an action ground state for $J_\lambda$ on $\R^N$, namely a soliton. But since $u$ vanishes in $\R^N \setminus \Omega$, this is impossible, unless $\Omega = \R^N$. The same argument shows that on such domains there is no energy ground state in $\MM_\mu$ for any $\mu$.

Thus, it makes no sense to discuss the problem when $\Omega$ contains balls of arbitrary radius. The class of  domains $\Omega$ that do not satisfy this property is quite large, and we consider here only a model case, to show that everything we said so far works also on some unbounded domains. Many more cases could be treated with essentially the same arguments as those we now outline.

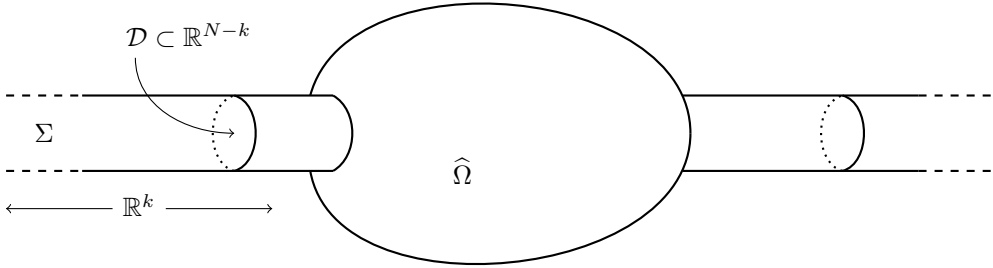
\begin{figure}[t]
	\centering
	\begin{tikzpicture}
	
	\draw[thick,-] (0,0)--(3.3,0);
	\draw[thick,-] (0,1)--(3.3,1);
	
	\draw[thick,dotted] (2,0) to [in= 195, out= 165 ](2,1);
	\draw[thick] (2,0) to [in= -15, out= 15 ](2,1);
	
	\draw[thick,dotted] (10,0) to [in= 195, out= 165 ](10,1); %
	\draw[thick] (10,0) to [in= -15, out= 15 ](10,1);
	
	\draw[thick] (3.3,0) to [in=-30, out=30] (3.3,1);

	\draw[thick,dashed] (-1,0)--(0,0);
	\draw[thick,dashed] (-1,1)--(0,1);
	
	\draw[thick,-] (7.9,0)--(11,0);
	\draw[thick,-] (7.9,1)--(11,1);
	\draw[thick,dashed] (11,0)--(12,0);
	\draw[thick,dashed] (11,1)--(12,1);
	
	\draw[thick] (3,0) to [in=-90, out=-80] (8,.5);
		
	\draw[thick] (3,1) to [in=90,out=75] (8,.5);
	
	\node at (1.4,1.8) [infinito] {$\DD\subset \R^{N-k}$};
	\draw[->] (.7,1.5) to[in=-180,out=-90](2,.5);
	
	\draw[->] (.4,-.5)--(-1,-.5);
	\draw[->] (1.1,-.5)--(2.5,-.5);
	\node at (.75,-.45) [infinito] {$\R^k$};
	
	\node at (5,0) [infinito] {$\widehat\Omega$};
	\node at (-0.5,0.5) [infinito] {$\Sigma$};
	
	\end{tikzpicture}
	\caption{Example of an unbounded domain $\Omega$ as in Section \ref{subsec:unbound}.}
	\label{fig:omega}
\end{figure}

In what follows, we take a bounded open set $\DD$ in $\R^{N-k}$ (for some $k   \in \{1,\dots, N-1\}$) and 
we consider the cylinder $\Sigma \subset \R^N$ defined as
\[
\Sigma = \DD \times \R^k.
\]
Finally, given a bounded  set $\widehat \Omega \subset \R^N$ such that $\widehat\Omega\cap\Sigma\neq\emptyset$ and $\widehat\Omega \setminus \Sigma \ne \emptyset$, we define  (see Figure \ref{fig:omega})
\begin{equation}
\label{omega}
\Omega:=\widehat{\Omega}\cup\Sigma.
\end{equation}
We begin by recalling some properties of the problem of action ground states on the cylinder $\Sigma$. To this aim we define
\[
J^\infty_\lambda (v) = \f12\|\nabla u\|_{L^2(\Sigma)}^2 + \f\lambda2 \|v\|_{L^2(\Sigma)}^2 -\frac1p\|u\|_{L^p(\Sigma)}^p,
\]
\[
\NL ^\infty= \left\{v\in H_0^1(\Sigma) \setminus \{0\}\; : \; \|\nabla v\|_{L^2(\Sigma)}^2 + \lambda\|v\|_{L^2(\Sigma)}^2 = \|v\|_{L^p(\Sigma)}^p\right\}
\]
and
\[
\JJ^\infty(\lambda):=\inf_{v\in\NL^\infty}J^\infty_\lambda(v).
\]
The following well--known result can be found for instance in \cite[Lemma 3]{Est83}.
\begin{lemma}
\label{bottom} Let $\lambda_1(\DD)$ be the first Dirichlet eigenvalue on $\DD$. Then
\[
\lambda_\Sigma :=\inf_{u\in H_0^1(\Sigma)}\frac{\|\nabla u\|_{L^2(\Sigma)}^2}{\|u\|_{L^2(\Sigma)}^2} = \lambda_1(\DD).
\]
\end{lemma}
Even though $\lambda_\Sigma$ is clearly not attained, the action ground state  level is continuous on $\R$. In view of Remark \ref{rem:lomega}, this follows from the next lemma.
\begin{lemma}
\label{lsigma}
	The action ground state  level on $\Sigma$ satisfies $\JJ^\infty(-\lambda_\Sigma)=0$.
\end{lemma}

\begin{proof} Let $\varphi_1\in H_0^1(\DD)$ be an eigenfunction associated to $\lambda_1(\DD) = \lambda_\Sigma$ and let $\psi \in C_0^\infty(\R^k) \setminus\{0\}$. For $\theta >0$ define $u_\theta \in H^1_0(\Sigma)$ by $u_\theta(x,y) = \psi(\theta x)\varphi_1(y)$. By elementary computations,
\[
\frac{\int_\Sigma |\nabla u_\theta|^2\dx\dy -\lambda_\Sigma \int_\Sigma |u_\theta|^2\dx\dy}{ \int_\Sigma |u_\theta|^p\dx\dy}
= \theta^2\frac{\int_{\R^k} |\nabla \psi|^2\dx \int_\DD |\varphi_1|^2\dy}{\int_{\R^k} | \psi|^p\dx \int_\DD |\varphi_1|^p\dy} = C\theta^2.
\]
Therefore
\[
0\leq\JJ^\infty(-\lambda_\Sigma)\leq\lim_{\theta\to 0} J_{-\lambda_\Sigma}^\infty(\sigma_{-\lambda_\Sigma}(u_\theta)u_\theta) = C\kappa\lim_{\theta\to 0}\theta^{\frac{2p}{p-2}} \|u_\theta \|_p^p = C\lim_{\theta\to 0}\theta^{\frac{2p}{p-2}-k} =0
\]
since the exponent of $\theta$ is positive for every $p < 2^*$.
\end{proof}

Concerning existence of action ground states on $\Sigma$, we have the following result.
\begin{theorem}
\label{cyl} For every $\lambda > -\lambda_1(\DD)$ there exists an action ground state for $J_\lambda^\infty$ in $\NN_\lambda^\infty$, namely there exists $u\in \NN_\lambda^\infty$ such that $J_\lambda^\infty(u) = \JJ^\infty(\lambda)$.
\end{theorem}

We do not provide a proof of this result since it can be carried out with the same arguments that work to prove the existence of action ground states in $\R^N$: compactness of minimizing sequences is recovered due to the invariance of the problem under translations along the directions in $\R^k$ (see e.g. \cite{cazenave,Est83} for details). 

We now focus on the domain $\Omega$ defined in \eqref{omega}. Also in this case, we first show that $\JJ$ is continuous on the whole real line. 

\begin{lemma}
The action ground state level on $\Omega$ satisfies $\JJ(-\lambda_\Omega)=0$. 
\end{lemma}
\begin{proof}
	Since $\Sigma\subset\Omega$, then $\lambda_\Omega\leq\lambda_\Sigma$. If $\lambda_\Omega=\lambda_\Sigma$, the lemma follows repeating verbatim the argument in the proof of Lemma \ref{lsigma}.
	
	If on the contrary $\lambda_\Omega<\lambda_\Sigma$, then we claim that there exists an eigenfunction $\varphi\in H_0^1(\Omega)$ associated to $\lambda_\Omega$. Note that, given this, the statement of the lemma follows by Remark \ref{rem:lomega}. Let us thus prove the claim. Let $(\varphi_n)_n\subset H_0^1(\Omega)$ satisfy $\|\varphi_n\|_2=1$ for every $n\in\N$ and $\|\nabla\varphi_n\|_2^2\to\lambda_\Omega$ as $n\to+\infty$. Then (up to subsequences) there exists $\varphi\in H_0^1(\Omega)$ such that $\varphi_n\rightharpoonup \varphi$ in $H^1(\Omega)$. Let $m:=\|\varphi\|_2^2$. By weak lower semicontinuity, $m\in[0,1]$. If $m=1$, then $\varphi$ is the eigenfunction we seek. We are thus left to rule out the case $m<1$.
	
Note first that $m\neq0$. Indeed, if this is not the case, then in particular  $\|\varphi_n\|_{L^2(\widehat{\Omega})}\to0$ as $n\to+\infty$. By a standard cut--off procedure, one can then  construct $u_n\in H_0^1(\Omega)$ such that $u_n\equiv0$ on $\widehat{\Omega}$, $\|u_n\|_2^2=1+o(1)$ and $\|\nabla u_n\|_2^2\leq\|\nabla \varphi_n\|_2^2+o(1)=\lambda_\Omega+o(1)$ as $n \to \infty$. But this is impossible, since as $u_n$ is supported in $\Sigma$,
\[
\lambda_\Sigma \le \frac{\|\nabla u_n\|_2^2}{\|u_n\|_2^2} \le \frac{\|\nabla \varphi_n\|_2^2+o(1)}{1+o(1)} = \lambda_\Omega +o(1),
\]
contradicting $\lambda_\Omega < \lambda_\Sigma$.

Hence $m\in(0,1)$ and $\varphi\not\equiv 0$ on $\Omega$. By weak convergence, we have $\|\varphi_n-\varphi\|_2^2= 1-\|\varphi\|_2^2+o(1)=1-m+o(1)$. Furthermore, $\varphi_n-\varphi\to0$ strongly in $L^2(\widehat{\Omega})$, so that just as before $\|\nabla(\varphi_n-\varphi)\|_2^2\geq\lambda_\Sigma\|\varphi_n-\varphi\|_2^2+o(1)=\lambda_\Sigma(1-m)+o(1)$ as $n \to \infty$. Hence
\[
\lambda_\Omega =\|\nabla\varphi_n\|_2^2 +o(1)=\|\nabla(\varphi_n-\varphi)\|_2^2+\|\nabla\varphi\|_2^2+o(1)\geq\lambda_\Sigma(1-m)+\lambda_\Omega m+o(1),
\]
contradicting again $\lambda_\Omega < \lambda_\Sigma$. The proof is complete.
\end{proof}

We now show that Assumption \ref{C} is fulfilled, and to this end we first need  to address existence of action ground states.
We have chosen to work with $\Omega$ since it is a first example where one cannot use directly the invariance under translations to restore compactness. It is clear that the general reason behind existence proofs resides in concentration--compactness arguments, of which the results below are no more than an adaptation to our specific setting. Since however the proofs are rather short, we carry them out for completeness. Furthermore, as we are specifically interested in the validity of Assumption \ref{C},  we limit ourselves to deal with action ground states. The same type of argument can be naturally adapted to prove existence of energy ground states, thus showing the validity of Assumption \ref{assE}.

The following preliminary lemmas characterize the behaviour of action minimizing sequences.
\begin{lemma}
\label{weakzero}
Let $(u_n)_n \subset \NL$ be a sequence such that $u_n \rightharpoonup 0$ in $H_0^1(\Omega)$ as $n\to \infty$. Then
\[
\liminf_n J_\lambda(u_n) \ge \JJ^\infty(\lambda).
\]
\end{lemma}

\begin{proof} Since each $u_n$ is in $H^1_0(\Omega)$, we can view it, whenever necessary, as an element of $H^1_0(\R^N)$, without repeating it every time. Moreover, recall that the sequence $u_n$, being in $\NL$, cannot tend to $0$ strongly in $L^p(\Omega)$.
	
Let $R>0$ be so that $\widehat \Omega \subset B_R$ and define $\phi : [0,+\infty) \to [0,1]$ as
\[
\phi(t) = \begin{cases} 0 & \text{ if } t\in [0,R] \\ 
t-R & \text{ if } t\in [R,R+1] \\
1 & \text{ if } t \ge R+1\,. \end{cases}
\]
Note that $|\phi'(t)| \le 1$ for every $t$.  Since $u_n \rightharpoonup 0$ in $H_0^1(\Omega)$, we also have $u_n \to 0$ in $L^q_{loc}(\Omega)$ for every $q \in [2, 2^*)$.

Therefore, defining  $v_n \in H_0^1(\Sigma)$ as (the restriction to $\Sigma$ of) $\phi(|x|)u_n(x)$, it is straightforward to check that as $n\to \infty$,
\begin{align*}
&\|v_n\|_{L^q(\Sigma)}^q = \|u_n\|_{L^q(\Omega)}^q + o(1)\,, \qquad q \in [2, 2^*),  \\
& \|\nabla v_n\|_{L^2(\Sigma)}^2 \le \|\nabla u_n\|_{L^2(\Omega)}^2  + o(1)\,.
\end{align*}
Then we see that 
\[
\|\nabla v_n\|_{L^2(\Sigma)}^2 + \lambda\|v_n\|_{L^2(\Sigma)}^2 \le \|\nabla u_n\|_{L^2(\Omega)}^2 + \lambda\|u_n\|_{L^2(\Omega)}^2 +o(1)
= \|u_n\|_{L^p(\Omega)}^p +o(1) = \|v_n\|_{L^p(\Sigma)}^p +o(1),
\]
or, as $\|v_n\|_{L^p(\Sigma)}^p$ is bounded away from zero,
\[
\sigma_\lambda(v_n)\le 1+o(1)\,.
\]
Notice now that, as $\sigma_\lambda(v_n) v_n \in \NL^\infty$,
\[
\JJ^\infty(\lambda) \le J_\lambda^\infty(\sigma_\lambda(v_n) v_n) = \kappa \sigma_\lambda(v_n)^p \|v_n\|_{L^p(\Sigma)}^p \le \kappa\|u_n\|_{L^p(\Omega)}^p +o(1) = J_\lambda(u_n) + o(1),
\]
and letting $n\to \infty$ the conclusion follows.
\end{proof}

\begin{lemma}
\label{weakstrong}
Assume that $\JJ(\lambda) < \JJ^\infty(\lambda)$.
Let $(u_n)_n \subset \NL$ be a minimizing sequence for $J_\lambda$  such that $u_n \rightharpoonup u$ in $H_0^1(\Omega)$ and a.e. in $\Omega$ as $n\to \infty$.  If $u \not\equiv 0$, then
\[
u_n \to u \quad\text{in}\quad L^2(\Omega)\qquad \text{as}\quad n \to \infty.
\]
\end{lemma}

\begin{proof} Assume by contradiction that 
\begin{equation}
\label{contr}
\liminf_n \|u_n - u\|_2^2 >0.
\end{equation}
Since $u_n\rightharpoonup u$ in $L^p(\Omega)$, we first see that
\[
\kappa \|u\|_p^p \le \kappa\liminf_n  \|u_n\|_p^p = \lim_n J_\lambda(u_n) = \JJ(\lambda) \le J_\lambda(\sigma_\lambda(u) u) = \kappa\sigma_\lambda(u)^p \|u\|_p^p,
\]
which shows (as $u\not \equiv 0$) that $\sigma_\lambda(u) \ge 1$.

Now by the Brezis--Lieb Lemma \cite{BL} we can write, as $n \to \infty$,
\begin{align*}
\lambda &= \frac{\|u_n\|_p^p -\|\nabla u_n\|_2^2}{\|u_n\|_2^2} \\
&= \frac{\|u_n-u\|_p^p -\|\nabla u_n-\nabla u\|_2^2 + \|u\|_p^p -\|\nabla u\|_2^2 +o(1)}{\|u_n-u\|_2^2 + \|u\|_2^2 + o(1)} =: \frac{a_n+b+o(1)}{c_n + d+o(1)}.
\end{align*}
Notice that $d \ne 0$ and that $b/d \le \lambda$, since $\sigma_\lambda(u) \ge 1$.
Therefore
\[
\lambda c_n + \lambda d +o(1) = a_n + b + o(1) \le a_n + \lambda d +o(1),
\]
namely $\lambda c_n \le a_n +o(1)$, which reads
\[
\sigma_\lambda(u_n-u) \le 1 +o(1)
\]
as $n \to \infty$ (notice that $\|u_n - u\|_p^p$ cannot tend to zero by $\lambda c_n \le a_n +o(1)$ combined with \eqref{contr}).

We now define $v_n = \sigma(u_n-u) (u_n-u) \in \NL$ for every $n$ and we notice that $v_n \rightharpoonup 0$ in $H_0^1(\Omega)$. By Lemma \ref{weakzero},
\[
\liminf_n J_\lambda(v_n) \ge \JJ^\infty(\lambda),
\]
so that (using $\sigma(u_n-u)  \le 1 +o(1)$) 
\begin{align*}
\JJ(\lambda) &= \kappa\lim_n  \|u_n\|_p^p = \kappa \lim_n \left(\|u_n-u\|_p^p + \|u\|_p^p\right) \\
& \ge \kappa \lim_n \sigma(u_n-u) ^p \|u_n-u\|_p^p + \kappa\|u\|_p^p \ge  \kappa \lim_n \|v_n\|_p^p = \lim_n J_\lambda(v_n) \ge \JJ^\infty(\lambda),
\end{align*}
which contradicts the assumption and concludes the proof.
\end{proof}
The next theorem provides existence for action ground states on $\Omega$.
\begin{theorem}
\label{exaction}
For every $\lambda \in (-\lambda_\Omega, +\infty)$  there exists an action ground state in $\NL$. 
\end{theorem}

\begin{proof} We first note that $\lambda_\Omega \le \lambda_\Sigma = \lambda_1(\DD)$, simply because $\Sigma \subset \Omega$ and by Lemma \ref{bottom}. Next we observe that for every $\lambda \in (-\lambda_\Omega, +\infty)$, there results
\begin{equation}
\label{ineq}
\JJ(\lambda) < \JJ^\infty(\lambda).
\end{equation}
Indeed, the weak inequality is trivial by the inclusion $\Sigma \subset \Omega$. If we had equality, then an action ground state in $\Sigma$ (provided by Theorem \ref{cyl}, recalling that $\lambda_\Omega \le \lambda_1(\DD)$) would be an action ground state in $\Omega$, which is impossible since it would vanish in $\Omega \setminus \Sigma$.

Let $(u_n)_n \subset \NL$ be a minimizing sequence for $J_\lambda$. As such, $(u_n)_n$ is bounded in $H^1_0(\Omega)$, and we can assume that (up to subsequences)
\[
u_n \rightharpoonup u \text{ in }H^1(\Omega), \quad u_n \to u  \text{ in $L^q_{loc}(\Omega)$ for every $q \in [2,2^*)$ and a.e. in $\Omega$}.
\]
Now, $u$ cannot vanish identically since in this case by Lemma \ref{weakzero}
\[
\JJ(\lambda) = \lim_n J_\lambda(u_n) \ge \JJ^\infty(\lambda),
\]
contradicting \eqref{ineq}. Hence, by Lemma \ref{weakstrong}, $u_n \to u$ strongly in $L^2(\Omega)$. By \eqref{GN} we also have $u_n \to u$ strongly in $L^p(\Omega)$. Hence
\[
\|\nabla u\|_2^2+\lambda \|u\|_2^2 \le \liminf_n \left(\|\nabla u_n\|_2^2 +\lambda \|u_n\|_2^2\right) = \liminf_n \|u_n\|_p^p = \|u\|_p^p,
\]
showing that $\sigma_\lambda(u)\le 1$. Now $\sigma_\lambda(u) <1$ is impossible, since if it were so,
\[
\JJ(\lambda) \le J_\lambda(\sigma_\lambda(u) u) = k \sigma_\lambda(u)^p \|u\|_p^p = k \sigma_\lambda(u)^p \lim_n \|u_n\|_p^p = \sigma_\lambda(u)^p \JJ(\lambda) < \JJ(\lambda),
\]
which is false. Hence it must be $\sigma_\lambda(u) = 1$, that is $u\in \NL$ which, coupled with
\[
J_\lambda(u) = \kappa \|u\|_p^p = \kappa \lim_n \|u_n\|_p^p = \JJ(\lambda),
\]
shows that $u$ is the required ground state.
\end{proof}
\begin{proposition}
\label{Cholds} 
Assumption \ref{C} holds. 
\end{proposition}
\begin{proof} 
The proof is analogous to that of Proposition \ref{compact}, the only difference being that we cannot rely on Sobolev embeddings to obtain strong compactness in $L^p(\Omega)$. Fix $\lambda >-\lambda_\Omega$, let $\lambda_n \to \lambda$ as $n \to \infty$ and $\mu_n \in Q(\lambda_n)$ be such that $\mu_n \to \mu$. As before, take an action ground state $u_n \in \NN_{\lambda_n}$ with $\|u_n\|_2^2 =2 \mu_n$ for every $n$, and let $u\in H_0^1(\Omega)$ be its weak limit in $H_0^1(\Omega)$ as $n\to+\infty$, so that $u_n\to u$ in $L_{loc}^q(\Omega)$, for every $q \in [2, 2^*)$. Since
\[
\sigma_\lambda(u_n) = \left(1 +(\lambda-\lambda_n)\frac{\|u_n\|_2^2}{\|u_n\|_p^p} \right)^{\frac1{p-2}} =1+o(1)
\]
as $n \to \infty$, by the continuity of $\JJ$ we have
\[
J_\lambda(\sigma_\lambda(u_n) u_n) = (1+o(1)) \kappa \|u_n\|_p^p= (1+o(1))J_{\lambda_n}(u_n) = (1+o(1))\JJ(\lambda_n) \to \JJ(\lambda),
\]
i.e. $\sigma_\lambda(u_n) u_n \in \NL$ is a minimizing sequence for $J_\lambda$.  Also, $\sigma_\lambda(u_n) u_n \rightharpoonup u$ in $H_0^1(\Omega)$. By Lemma \ref{weakzero}, $u$ cannot vanish identically (since we would have $\JJ(\lambda) \ge \JJ^\infty(\lambda)$), and by Lemma \ref{weakstrong}, $\sigma_\lambda(u_n) u_n$, and hence $u_n$, converges strongly to $u$ in $L^2(\Omega)$, and then in $L^p(\Omega)$ by \eqref{GN}. As in the proof of Theorem \ref{exaction}, this shows that $u$ is an action ground state in $\NL$ and, arguing as in the last part of the proof of Proposition \ref{compact}, we conclude.
\end{proof} 

\begin{remark}
\label{rem:C}
Note that, both for bounded domains and for the unbounded set above, one can easily refine the previous arguments to show that the set of action ground states in $\NL$ is strongly compact in $H_0^1(\Omega)$. It is clear that compactness for every $\lambda$ guarantees the validity of Assumption \ref{C}. However, there are cases where one can obtain Assumption \ref{C} even though the set of action ground states is not compact. For example, if $\Omega = \R^N$, then the set of action ground states in $\NL$ is the family of solitons, which is not compact. Nevertheless, it is immediate to see (and well known) that Assumption \ref{C}  holds. Another example along these lines is the cylinder $\Sigma$.
	
\end{remark}

\section*{Acknowledgements}
The work has been partially supported by the MIUR project ``Dipartimenti di Eccellenza 2018--2022" (CUP E11G18000350001) and by the INDAM-GNAMPA project 2020 ``Modelli differenziali alle derivate parziali per fenomeni di interazione".

\end{document}